\documentclass[11pt,reqno]{article}

\usepackage[margin=1in]{geometry}
\usepackage{amsthm,amsmath,epsfig,subfigure,amsfonts,amssymb,bm,xcolor,amsbsy,times}
\usepackage{bm}
\usepackage{bbm}
\usepackage{enumerate}
\usepackage{mathrsfs}
\usepackage{graphicx}
\usepackage{comment}
\usepackage{mathtools}
\usepackage{cases}
\usepackage{dsfont, amssymb,amsmath,amscd,latexsym, amsxtra,amsfonts,amsthm}
\numberwithin{equation}{section}
\graphicspath{{figures/}}
\usepackage{ifpdf}
\usepackage{color}
\definecolor{webgreen}{rgb}{0,.5,0}
\definecolor{webbrown}{rgb}{.8,0,0}
\definecolor{emphcolor}{rgb}{0.95,0.95,0.95}

\usepackage{hyperref}
\hypersetup{%
          colorlinks=true,
          linkcolor=blue,
          filecolor=blue,
          citecolor=blue,
          breaklinks=true}

\usepackage[round]{natbib}
\newtheorem{theorem}{Theorem}[section]
\newtheorem{corollary}[theorem]{Corollary}
\newtheorem{proposition}[theorem]{Proposition}
\newtheorem{lemma}[theorem]{Lemma}

\theoremstyle{definition}
\newtheorem{definition}[theorem]{Definition}

\theoremstyle{definition}

\theoremstyle{definition}
\newtheorem{remark}{Remark}
\theoremstyle{definition}
\newtheorem{assumption}{Assumption}
\newcommand{\md}{\mathrm{d}}

\newcommand{\mif}{\mathrm{if}}
\newcommand{\ow}{\mathrm{otherwise}}
\newcommand{\mR}{\mathbb{R}}

\newcommand{\mN}{\mathbb{N}}

\newcommand{\mE}{\mathbb{E}}
\newcommand{\mF}{\mathbb{F}}
\newcommand{\mP}{\mathbb{P}}

\newcommand{\mK}{\mathbb{K}}
\newcommand{\mX}{\mathbb{X}}
\renewcommand{\epsilon}{\varepsilon}

\newcommand{\F}{\mathcal{F}}

\newcommand{\Pc}{\mathcal{P}}

\newcommand{\Oi}{\mathcal{O}}
\newcommand{\B}{\mathcal{B}}
\newcommand{\Ui}{\mathcal{U}}

\newcommand{\cA}{\mathcal{A}}

\newcommand{\bu}{\mathbf{u}}
\newcommand{\bU}{\mathbf{U}}

\newcommand{\ddo}{\mathrm{\ddot{o}}}

\newcommand{\intt}{\mathrm{int}}
\newcommand{\hato}{\hat{\mathrm{o}}}
\newcommand{\sgn}{\mathrm{sgn}}
\newcommand{\mt}{\mathbf{t}}
\newcommand{\mB}{\mathbf{B}}
\newcommand{\ubeta}{\underline{\beta}}
\newcommand{\obeta}{\overline{\beta}}

\title{Weak equilibria for time-inconsistent control: with applications to investment-withdrawal decisions}

\author{Zongxia Liang\thanks{
		Department of Mathematical Sciences, Tsinghua University, Beijing, China, email: \texttt{liangzongxia@tsinghua.edu.cn}}	
	\and Fengyi Yuan\thanks{
		Department of Mathematical Sciences, Tsinghua University, Beijing, China, email: \texttt{yfy19@mails.tsinghua.edu.cn}}}	

\begin{document}

\maketitle
\begin{abstract}
This paper considers time-inconsistent problems when control and stopping strategies are required to be made {\it simultaneously} (called stopping control problems by us). We first formulate the time-inconsistent stopping control problems under general multi-dimensional controlled diffusion model and propose a formal definition of their equilibria. We show that an admissible pair $(\hat{u},C)$ of control-stopping policy is equilibrium if and only if the auxiliary function associated with it solves the extended HJB system, providing a methodology to {\it verify} or {\it exclude} equilibrium solutions. {  We provide several examples to illustrate applications to mathematical finance and control theory. For a problem whose reward function endogenously depends on the current wealth, the equilibrium is explicitly obtained. For another model with a non-exponential discount, we prove that any constant proportion strategy can not be equilibrium. We further show that general non-constant equilibrium exists and is described by singular boundary value problems. This example shows that considering our combined problems is essentially different from investigating them separately. In the end, we also provide a two-dimensional example with a hyperbolic discount. }
\end{abstract}

\section{Introduction}
In economics and finance, it is common that people make control and stopping decisions simultaneously. When making investment decisions, they are usually free to terminate the investment discretionarily; when participating the gambling, besides how much to bet, when to exit the casino is also an important decision to make; when provided with a voluntary early retirement option, they shall simultaneously decide how much to consume before (and after) retirement and when to retire.

In this paper, under the general multi-dimensional controlled diffusion model, we provide a solution theory for the stopping control problems with pay-offs of the form
\[
J(t,x;u,\tau)\triangleq \mE^{t,x} g(t,x,\tau,X^{u}_{\tau}).
\]
Unlike stopping control problems in existing literature, our problems are generally {\it time-inconsistent} because the reward function $g$ depends on the starting states $(t,x)$. Time-inconsistency appears when the optimal control or stopping policies at each starting state do not match with each other, in which case the dynamic programming principle is invalid and the problem to find an ``dynamic optimal control" is generally ill-posed. In this paper, we formulate the time-inconsistent problem as an intra-person game and propose the notion of equilibrium solution when stopping and control strategies are coupled together. Although time-inconsistent stopping or control problems have been widely studied in recent years, this paper is the first step towards the coupled problems along with time-inconsistency to the best of our knowledge. Several difficulties arise immediately.

There are several notions of equilibrium in existing literature, such as strong, weak, and mild equilibrium. We will discuss them in detail in Subsection \ref{LR}. For the present paper, the first question is to choose the notion of equilibrium on which we work. The strong equilibrium formulation is at its early stage and is found to be rather difficult to characterize. Strong equilibrium may not exist in some simple examples (see Section 4.4 in \citet*{He2019}). Mild equilibrium, on the other hand, only makes sense when we consider stopping (or discrete-valued) policies. It is not clear for now what the mild equilibrium for the control problem is. Therefore we choose to work on weak equilibrium formulation. However, it is not trivial to extend the notion of weak equilibrium to the context of stopping control. If stopping control strategies are considered together, one disadvantage is that these two parts of strategies influence each other and hence are fully coupled. To address this, we choose to work on weaker game-theoretical reasoning that ``unilateral deviation of stopping or control strategy {\it separately} does not improve utility". We emphasize that the proposed definition not only brings tractability for deriving associated extended HJB systems but is also consistent with the weak equilibrium formulations for pure control/stopping problems in existing literature. {  For completeness, we also briefly discuss the definition where the stopping and control strategies are allowed to deviate simultaneously. We prove that the latter is strictly stronger than the former and also explain why we do not concentrate on this stronger notion (see Definition \ref{stricteqdef} and Remark \ref{counterexm}).}

The second difficulty is the derivation of the associated extended HJB system. One of the main contributions of \citet*{Bjork2017} is to characterize the equilibrium value and {\it find} equilibrium strategies via extended HJB equations. The tractability of finding the equilibrium strategies is considered to be one of the main superiorities of considering solely the control problem. Therefore, establishing the associated extended HJB system when stopping and control are coupled together will be the central task of the research on this topic. Using cut-off and localization techniques and the properties of characterization operators of transition semigroups, we rigorously establish the associated extended HJB system, which does not only contain equations but also inequalities as well as a complicated boundary term. Combining the above arguments with parabolic PDE theory, we can weaken the conditions in \citet*{He2019}. We emphasize that the derived extended HJB system is an equivalent characterization of equilibria, which is thus beyond the verification scope. Moreover, the extended HJB system established in this paper is a nontrivial generalization of both the extended HJB equations in \citet*{Bjork2017} and the time-inconsistent variational inequalities in \citet*{Christensen2018}.

The last difficulty is the further characterization of the boundary term (\ref{sys3}) in the extended HJB system. This term is missed in pure control problems and appears as the celebrated smooth fitting principle in pure stopping problems. However, in pure stopping problems, smooth fitting is usually just a sufficient condition of equilibrium (see e.g. \citet*{Christensen2018}). In \citet*{Christensen2020}, it is proved that smooth fitting is also necessary for equilibrium, under a one-dimensional diffusion model. Using cut-off techniques mentioned above, we apply Peskir's generalized local time formula on (hyper-) surfaces (see e.g. \citet*{Peskir2007}) to prove that smooth fitting is in some sense both sufficient and necessary for the boundary term (\ref{sys3}) in extended HJB system. { The sufficiency part is useful for verification procedures (see our examples in Section \ref{exm}). The necessity part, on the other hand, states that the equilibrium value function is global $C^1$ along spatial directions. Global regularity of value functions is a very important topic in optimal stopping theory (see \citet*{DEANGELIS2020} and literature review therein), and we have obtained similar results in the time-inconsistent counterpart, yet based on stronger regularity assumptions on stopping boundary.}

We extend all of our results to the infinite time horizon case, and {  thoroughly investigate three examples, illustrating the applications of our theoretical framework to some practical financial models or interesting mathematical problems. Our examples include}:
\vspace{-5pt}
\begin{itemize}
	\item [(1)]An investment-withdrawal decision model with endogenous habit formation. Constant investment proportion with a one-side withdrawal threshold can still be equilibrium in this case, but this is only assured when the dependence on habit is not very ``strong" (see Proposition \ref{exmconclusion} for details). This strategy structure is similar to those obtained in existing literature, such as time-consistent stopping control problems \citet*{Karatzas2000}, pure stopping problems \citet*{Christensen2018} or pure control problems \citet*{He2019}. However, our numerical experiments show that combining the control and stopping strategies brings novel financial insights. For example, the withdrawal threshold will decline when the market risk is higher. See Subsection \ref{example1}
	\item [(2)]An investment-withdrawal decision model with logarithm utility and ambiguity on discount factor. In this example, constant investment proportion can not constitute equilibrium, no matter what the withdrawal strategy is. This phenomenon is rare in the existing literature, because if the desired solution can be obtained explicitly, the investment strategy is usually constant proportion (or time-dependent but wealth-independent proportion, if the time horizon is finite)\footnote{See \citet*{Karatzas2000}, \citet*{Ekeland2008}, \citet*{Yong2012},  \citet*{Bjork2017}, \citet*{Alia2017} and \citet*{He2019}, among others.}. { Moreover, with a two-point quasi-exponential discount, we further show that the equilibrium does exist, thus it must be a non-constant one. We find that in this case the control and stopping { parts} of the equilibrium pair are completely coupled and are described by two highly nonlinear singular boundary value problems (see (\ref{sBVP1})-(\ref{sBVP2})). They are themselves very interesting mathematical objectives and we show the existence of their solutions based on the cut-off technique, Leray-Schauder degree theory, and Green function (see Appendix \ref{extsBVP}). This example shows that combining stopping control problems with time-inconsistency leads to highly nontrivial and challenging problems. It also provides a theoretical counter-view against constant proportion investment suggested by Merton's theory or classical stopping-control problems even in the simplest market.} See Subsection \ref{nonexp}. { This example also constitutes the main contributions of this paper, both mathematical and financial.}
	{ 	
		\item[(3)]A stopping control decision problem about planar Brownian Motion, where the agent determines diffusion coefficient and stopping radius. The reward is a hyperbolic discount factor multiplied by the stopping radius, thus there is a trade-off between a larger radius and less discount. Using our theoretical framework, the rational strategy is to diffuse the system as much as possible, and choose a stopping radius that is proportional to $1/\sqrt{\beta}$, where $\beta$ is the discount rate. The value of this radius is related to a universal constant ($\approx$ 8.3419) determined by Bessel functions. This example is interesting in the aspect of mathematics and is the first step to further studies on time-inconsistent multi-dimensional stopping control problems. See Subsection \ref{2dexm}.
	}
\end{itemize}
\vspace{-5pt}
To conclude, the main contributions of the present paper are as follows:
\vspace{-5pt}
\begin{itemize}
	\item[(1)]We establish the framework for studying time-inconsistent stopping control problems under weak equilibrium formulation, and the proposed formulation is a nontrivial extension for weak formulations in pure stopping/control problems.
	\item [(2)]Using cut-off and localization techniques, we rigorously obtain an equivalent characterization of the equilibrium, which is an extended HJB system, and the assumptions needed are weakened. We believe that the extended HJB system established in this paper will become the foundation of future research on time-inconsistent stopping control problems in more specific models such as investment with discretionary stopping under non-exponential discount.
	\item [(3)]We build connections between the aforementioned HJB system and the smooth fitting principles in optimal stopping theory, and {  show that smooth fitting is in some sense sufficient and necessary for equilibrium conditions}.
	\item [(4)]{ We demonstrate two concrete investment-withdrawal problems shedding light on people's behaviors facing time-inconsistent preference, which is common in behavioral finance. One of these examples indicates that combining stopping control problems with time-inconsistency brings us essential differences and produces highly nontrivial problems. Finally, we also provide a multi-dimensional example controlling planar Brownian motion, which is interesting mathematically.}
\end{itemize}
\vspace{-5pt}
The rest of the paper is organized as follows: We formulate the time-inconsistent stopping control problem and define the equilibrium strategies in Section \ref{notation}. Section \ref{eqchar} and Section \ref{exm} include the main results of this paper. In Section \ref{eqchar}, we obtain characterizations of the equilibrium. In Section \ref{exm}, we provide three concrete examples to illustrate the theoretical results in Section \ref{eqchar}. Section \ref{conlude} concludes this paper. Technical proofs are mainly presented in Appendix \ref{proofsec3}. In Appendix \ref{fconddiscussion}, we discuss how to use PDE theory to obtain our assumptions in Section \ref{eqchar}. Appendices \ref{proofexmcon} and \ref{extsBVP} include some technical results related to our examples.

\subsection{Literature Review} \label{LR}
Classical stopping control problems are extensively studied, and they are typically {\it time consistent}, which means that at any given initial time and state, the agent conducts the static optimization and the resulting strategies are consistent between different starting states. \citet*{Karatzas2000} and \citet*{Karatzas2006} study stopping control problems in a specific setting of portfolio choice. Interpreting this as a cooperative game between ``controller" and ``stopper", \citet*{Karatzas2001}, \citet*{Karatzas2008}, \citet*{Bayraktar2013} and \citet*{Bayraktar2019}, among others, study the noncooperative version of stopping control problems. Stopping control modelings are also widely applied to retirement decisions research, see \citet*{Choi2008}, \citet*{Farhi2007}, \citet*{Dybvig2009}, \citet*{Jeon2020}, \citet*{Xu2020} and \citet*{Guan2020}.


Although originated in the early but seminal paper \citet*{strotz1955}, game-theoretical formulations for time-inconsistent problems in continuous time have been taken into researchers' sight only in decades. Earlier developments focus on control problems, where the state dynamics are controlled and the terminal times are fixed. In some specific settings such as portfolio management or economic growth, \citet*{Ekeland2008} and \citet*{Ekeland2010} propose the notion of equilibrium in continuous time problems and provide existence. Among many other works on the same topic, one breakthrough is that \citet*{Bjork2017} derives the necessary conditions of equilibrium under the general Markovian diffusion model, which is an extended HJB equation system. This formulation of equilibrium is usually called weak equilibrium in the literature, and, roughly speaking, reveals the game theoretical reasoning ``{\it unilateral deviation does not improve the pay-off}" in a weak sense. Formally, $\hat{u}$ is said to be a weak equilibrium if
\[
\limsup_{\epsilon\to 0}\frac{J(t,x;u_{\epsilon})-J(t,x;\hat{u})}{\epsilon}\leq 0,\ \ \forall (t,x),
\]
where $u_{\epsilon}$ is properly-defined perturbation of $\hat{u}$.

\citet*{He2019} rigorously proves the extended HJB equation system for weak equilibrium formulation. \citet*{Hu2012} and \citet*{Hu2017} study the stochastic linear-quadratic control in similar formulation, and importantly, they obtain a uniqueness result for this particular control model. More recently  \citet*{Hu2020} obtains an analytical solution to a portfolio selection problem with the presence of probability distortion. In \citet*{He2019} and \citet*{Huang2020a}, the authors extend the scope of research on this topic and propose different notions of equilibria. An important one of those newly proposed notions is strong equilibrium, which requires $\forall (t,x), \exists \epsilon_0>0$,
\[
J(t,x;u_{\epsilon})-J(t,x;\hat{u})\leq 0,\ \ \forall \epsilon < \epsilon_0.
\]
This notion is closer to the concept of subgame perfect equilibrium (SPE) in game theory, but is rather difficult to study, as revealed in \citet*{He2019} and \citet*{Huang2020a}. Most recently, \citet*{Hernandez2020} studies time-inconsistent control problems in the most general non-Markovian framework, where the notion of equilibrium is also refined.

Stopping decision problems with time-inconsistency, on the other hand, are at the early stages of research. In recent years there are substantial works and theoretical breakthroughs on this topic. Series of papers including \citet*{Huang2018}, \citet*{Huang2020}, \citet*{Bayraktar2020}, \citet*{Huang2020b}, \citet*{Huang2020c} and \citet*{Huang2019} study time-inconsistent stopping problems under the so-called ``mild equilibria" ( \citet*{Bayraktar2020}), which only considers the deviation from the strategy ``stopping" to ``continuation": $C$ is an equilibria stopping policy (continuation region) if
\[
J(t,x;{\rm stopping}) \leq J(t,x;{\rm continuation}), \forall (t,x)\in C.
\]
We also note that there are indeed papers adopting weak equilibrium formulations when studying stopping problems, such as \citet*{Christensen2018} and \citet*{Christensen2020}, where not only deviation from ``stopping" to ``continuation", but also any infinitesimal perturbation of stopping policies are considered (see Definition \ref{eqdef} and remarks therein for detailed discussion). { In the context of time-inconsistent stopping problems, \citet*{Bayraktar2022} studies the relationships among different notions of equilibrium solutions. Similar to our paper, they provide full characterizations of weak equilibrium (under a one-dimensional diffusion model). Moreover, they prove that under some conditions, optimal mild equilibrium is strong, so that the existence of strong (weak) equilibrium can be obtained.} \citet*{Ebert2020} investigates weak equilibrium of time-inconsistent stopping problems from an economic point of view, and \citet*{Tan2021} thoroughly studies the smooth fitting principle with the presence of time-inconsistency, which is also one of our concerns. {In addition to the {\it pure} stopping rules which consist of the mainstream of current research,  \citet*{Bodnariu2022} focuses on the local time pushed {\it mixed} stopping rules and also provide smooth fitting principle under their formulation. Interestingly, they find that there is sometimes a dichotomy between the existences of pure and local time pushed mixed stopping rules.}

\section{Notations and model formulations}
\label{notation}
We first introduce key ingredients for our problem formulation:
\begin{itemize}
  \item {\bf State space}: Let $\mX\subset \mR^n$ be a region, equipped with the Euclidean norm $\|\cdot\|$ and $\B(\mX)$, the Borel $\sigma$-algebra induced by it. $E=[0,T)\times \mX$ will be the (space-time) state space.
  \item {\bf Control space}: Let $\bU$ be a Polish space, equipped with $\B(\bU)$, the Borel $\sigma$-algebra induced by its metric. We assume that the control takes value in $\bU$.
  \item {\bf Set of admissible controls}: Let $\Ui$ be a subset of all measurable closed-loop control. In other words, any $u\in \Ui$ is a measurable map $E\ni (t,x)\mapsto u(t,x)\in \bU$. Some assumptions will be imposed on $\Ui$ to guarantee the well-posedness of the problem (see Assumption \ref{assumption1}).
  \item {\bf Set of admissible stopping policies}: Let $\Oi$ be a subset of the collection of relative open subsets of $E$, including $\varnothing$ and $E$. $\Oi$ will be considered as the set of admissible continuation region, hence can be equivalently thought as the set of stopping policies.
  \item {\bf The implemented stopping time}: $\tau_{(u,C,s)}\triangleq \inf\{s\leq r\leq T: (r,X^u_r)\notin C          \}$. After the agent has chosen the {\it control} $u\in \Ui$, the {\it stopping policy} $C\in \Oi$, and the {\it minimal stopping time} $s$, he will implement $\tau_{(u,C,s)}$ as his stopping strategy. In other words, he will not consider stopping before $s$, and after time $s$, he will stop immediately when the state process under $u$ exits the continuation region $C$ for the first time.
  \item {\bf Set of admissible strategies at time $t$}: $\cA(t) \triangleq \{(u,\tau_{(u,C,s)}): u\in \Ui, s\geq t, C \in \Oi\}$.
\end{itemize}
In addition, the mathematical notations that will be used are listed below for convenience:
\begin{itemize}
  \item $C^{1,2}(\Omega)$ is the set of the functions that are once continuously differentiable in $t$ and twice continuously differentiable inside some subset $\Omega\subset E$.
  \item $C^0(\Omega)$ is the set of continuous functions.
  \item $C^{\infty}_c(\Omega)$ is the set of smooth functions with compact support.
  \item $d((s,y),(t,x))=|s-t|^{1/2}+\|x-y\|$ is the parabolic distance. $d((s,y),A)=\inf\{d((s,y),(s',y')):(s',y')\in A\}$.
  \item $B_{t,x}(\delta)=(t,t+\delta^2)\times \{x':\|x'-x\|<\delta\}$ is the parabolic cylinder.
  \item $\partial \Omega \in C^2$ means that for any $(t,x)\in \partial \Omega$ there exist $\delta>0$ which is sufficiently small, and a function $P\in C^{1,2}(B_{t,x}(\delta))$ such that $\{(s,y)\in B_{t,x}(\delta):(s,y)\in \partial \Omega\}=\{(s,y)\in B_{t,x}(\delta):P(s,y)=0\}$.
\end{itemize}
The probability basis for our problem is as usual. Let $(\Omega,\F,\mF,\mP)$ be a complete probability space supporting a standard $n$-dimensional Brownian motion $W=\{W_t:0\leq t \leq T\}$, where $\mF = \{\F_t:0\leq t \leq T\}$ is a filtration satisfying the usual conditions, and $\F = \F_T$\footnote{Here we are not concerned with what the filtration is. We choose to restrict the information by allowing only the Markovian closed-loop strategies, instead of specifying the filtration. See Remark \ref{info} for detailed discussion}. For any $u\in \Ui$, we consider the {\it strong formulation} of the controlled state dynamics:
\begin{equation}\label{dynamic}
  \left\{
  \begin{aligned}
  &\md X^u_t = \mu(t,X^u_t,u(t,X^u_t))\md t + \sigma(t,X^u_t,u(t,X^u_t))\md W_t,\\
  & X^u_0=x_0\in \mX.
  \end{aligned}
  \right.
\end{equation}
For any $h(\cdot,\cdot)\in C^{1,2}(B_{t,x}(\delta))$, $\delta>0$, we define the characterization operator $A^u$ of $X^u$ by\footnote{We denote by $M^{\mt}$ the transpose of the matrix $M$.}
\[
A^{u}h(t,x)\triangleq h_t(t,x)+\Theta^u(t,x)^{\mt}h_x(t,x)+\frac{1}{2}{\rm tr}(h_{xx}(t,x) \Lambda^u(t,x)\Lambda^u(t,x)^{\mt}),
\]
where
\begin{align*}
&\Theta^{u}(t,x)\triangleq \mu(t,x,u(t,x)),\\
&\Lambda^{u}(t,x)\triangleq \sigma(t,x,u(t,x)).
\end{align*}
We make the standing assumption on the admissible set $\Ui$, which is crucial for the analysis later, as follows:
\begin{assumption}
\label{assumption1}For any $u\in \Ui$, we have:
  \begin{itemize}
    \item[(1)] $\bU\subset \Ui$.\footnote{Here we identify any $\bu \in \bU$ with the constant map $(t,x)\mapsto \bu$. Therefore any element in $\bU$ can also be seen as a closed-loop control.}
    \item[(2)] $\Theta^u$ and $\Lambda^u$ are Lipschitz in $x\in \mR^n$, uniformly in $t$.
    \item[(3)] $\Theta^u$ and $\Lambda^u$ are right continuous in $t\in [0,T]$, and are bounded in $t$, uniformly for $x$ in any compact subsets of $\mR^n$.
    \item[(4)] For any $(t,x)\in E$, the solution of (\ref{dynamic}) with initial condition $X_t=x$, denoted by $X^{u,t,x}$, satisfies $\mP(X^{u,t,x}_s\in \mX, \forall\ 0\leq s\leq T)=1$.
  \end{itemize}
\end{assumption}
Based on the theory of stochastic differential equations (see e.g., \citet*{Friedman1975} or \citet*{Yong1999}), under Assumption \ref{assumption1}, (\ref{dynamic}) is well-posed in the strong sense, and the solution is strong Markovian, for any $u\in \Ui$. We introduce the Markovian family $ \{\mP^{t,x}\}_{(t,x)\in E}$ with $\mP^{0,x_0}=\mP$.

We now state our problem. The agent aims to maximize the pay-offs $J(t,x;u,\tau)$ among all admissible strategies $(u,\tau)\in \cA(t)$, which is given by
\[
J(t,x;u,\tau)\triangleq \mE^{t,x} g(t,x,\tau,X^{u}_{\tau}).
\]
However, due to the dependence of the reward on state, the problem is generally time-inconsistent, and it does not make sense to find the dynamic ``optimal" strategies\footnote{For the discussion on time-inconsistency for pure control problems, see \citet*{Ekeland2008} or \citet*{Yong2012}. For discussions on stopping problems, refer to \citet*{Huang2020}.}.  As in \citet*{Bjork2017} and \citet*{He2019}, we consider the perturbed control of $u\in \Ui$ by $\bu \in \bU, \epsilon >0$, which is defined as
\begin{equation}
\label{perturbation}
u_{(t,\epsilon,\bu)}(s,y) = \left\{
\begin{aligned}
&\bu,  \ \ s\in[t,t+\epsilon),   \\
&u(s,y),\ \ s\notin [t,t+\epsilon).
\end{aligned}
\right.
\end{equation}
We are now ready to give the definition of equilibrium policies of time-inconsistent stopping control problems.
\begin{definition}
\label{eqdef}
$(\hat{u},C)\in \Ui\times \Oi$ is said to be an equilibrium if and only if all of the followings hold with $\hat{\tau}\triangleq \tau_{(\hat{u},C,t)}$:
\begin{align}
  & g(t,x,t,x)\leq J(t,x;\hat{u},\hat{\tau}),\ \ \forall (t,x)\in E, \label{eqcon1}\\
  & \limsup_{\epsilon \to 0}\frac{J(t,x;\hat{u},\tau_{(\hat{u},C,t+\epsilon)})-J(t,x;\hat{u},\hat{\tau})}{\epsilon} \leq 0, \ \ \forall (t,x)\in E, \label{eqcon2}\\
  & \limsup_{\epsilon \to 0}\frac{J(t,x;\hat{u}_{(t,\epsilon,\bu)},\tau_{(\hat{u}_{(t,\epsilon,\bu)},C,t)})-J(t,x;\hat{u},\hat{\tau})}{\epsilon} \leq 0,\ \ \forall (t,x)\in C, \bu \in \bU. \label{eqcon3}
\end{align}
\end{definition}
\begin{remark}
  This paper provides a unified theory to investigate the stopping and control problems. Noting that $\Ui$ and $\Oi$ are only required to be some subset of universal {\bf feasible} action space, our work generalizes previous literatures. Indeed, take $\Ui=\{u_0\}$ to be the singleton, the problem degenerates to pure stopping problem, as in \citet*{Christensen2018}. Take $\Oi=\{E\}$, the problem degenerates to pure control problems, as in \citet*{Bjork2017} and \citet*{He2019}.
\end{remark}
\begin{remark}
\label{info}
 It is crucial to specify how much information the agent can make use of. Indeed, limited information is one important reason for the occurrence of time-inconsistency. Most literatures impose the limitation on filtration: the control process is required to be measurable with respect to the filtration generated by state process. In the present setting of (Markovian) controlled-diffusion model, we choose to impose similar limitation implicitly by allowing only closed-loop control. This is the analogy of considering only pure Markovian stopping times when studying time-inconsistent stopping problems, as in \citet*{Christensen2018}, \citet*{Huang2018}, among others.
\end{remark}
\begin{remark}
\label{rmk1}
   If $(t,x)\notin C$, (\ref{eqcon1}) becomes equality and hence trivial. If $(t,x)\in C$, (\ref{eqcon1}) states that it is better to continue than to stoping. However, $(t,x)\in C$ implies that the equilibrium policy commands the agent to continue. As such, (\ref{eqcon1}) states that the agent has no reason to deviate the equilibrium stopping policy from continuation to stopping. (\ref{eqcon2}) requires that the agent is not willing to deviate even an infinitesimals from equilibrium stopping policy. In conclusion, our definition requires the agent not to deviate from equilibrium stopping policy, if all the $(t,x)$-agents follow the equilibrium.
\end{remark}
\begin{remark}
\label{rmk2}
  (\ref{eqcon3}) states that if all the $(t,x)$-agents follow the equilibrium (control and stopping policy), then he is not willing to deviate from equilibrium control policy. Here the $(t,x)$-agents with $(t,x)\notin C$ is irrelevant because they have stopped and exited the system, as the equilibrium stopping policy commands.
\end{remark}
\begin{remark}
It should be noted that in Remarks \ref{rmk1} and \ref{rmk2}, the game-theoretical concept of equilibrium is in a weak sense. Indeed, when the limit in (\ref{eqcon2}) or (\ref{eqcon3}) is 0, then deviation from equilibrium may indeed improve the preference level. However, as understood in \citet*{Hernandez2020}, if we ignore those improvements that are as small as a proportion of length of time interval perturbed, the definition of (weak) equilibrium fits into the game-theoretical consideration. For other types of equilibrium, see \citet*{Huang2020a}, \citet*{Bayraktar2020} for examples. In the context of time-inconsistent stopping, different types of equilibrium have also been studied. \citet*{Christensen2018} and \citet*{Christensen2020} consider the weak equilibrium, which we adopt. Generally it is difficult to characterize strong equilibrium, even for pure stopping or control problems, see \citet*{Huang2020a} for this under the setting of discrete Markov chain. It shall be an interesting topic for future work to characterize strong equilibrium for continuous-valued state process, in pure stopping and control problems, as well as stopping-control problems.
\end{remark}
\begin{remark}
  Note that $\tau_{(u,C,t+\epsilon)}=\epsilon + \tau_{(u,C,t)}\circ \theta_{\epsilon}$, where $\{\theta_t\}_{t\geq 0}$ is the family of shift operators. Therefore, comparing to \citet*{Christensen2018}, the infinitesimal perturbation of $\tau$ is in the time horizon $\epsilon$, not in the space horizon $\tau_{\epsilon}=\inf\{s\geq 0:|X_s-X_0|\geq \epsilon\}$. In our case, to make two perturbations of control and of stopping consistent with each other, we consider perturbation in time. Note that the perturbation of control can not be in space, because that would destroy the Lipschitz property of $\Theta^{u}$ and $\Sigma^{u}$ so the SDE could be ill-posed.
\end{remark}
{  In Definition \ref{eqdef}, the possible deviations of control and stopping policies are separable, i.e., deviating from stopping policies when fixing control policies, and deviating from control policies when fixing stopping policies. This definition makes perfect sense from game theoretical perspective if the problem is understood as a cooperative stopper-controller problem, where one agent controls the system and another chooses when to terminate. For the one-agent setting in this paper, it is very natural to consider concept of equilibrium where deviations of policies are allowed to be mixed. Here we introduce the following definition of {\it strict} weak equilibrium.
\begin{definition}
	\label{stricteqdef}$(\hat{u},C)\in \Ui\times \Oi$ is said to be a strict weak equilibrium if and only if the followings hold with $\hat{\tau}=\tau_{(\hat{u},C,t)}$:
	\begin{align}
		& g(t,x,t,x)\leq J(t,x;\hat{u},\hat{\tau}),\ \ \forall (t,x)\in E, \label{stricteqcon1}\\
		& \limsup_{\epsilon \to 0}\frac{J(t,x;\hat{u}_{(t,\epsilon,\bu)},\tau_{(\hat{u}_{(t,\epsilon,\bu)},C,t+\epsilon)})-J(t,x;\hat{u},\hat{\tau})}{\epsilon} \leq 0,\ \ \forall (t,x)\in E, \bu \in \bU. \label{stricteqcon2}
	\end{align}
\end{definition}
In (\ref{stricteqcon1}) we capture the deviation of policies when the stopping part is ``from continuing to stopping", and in this case the control policy is absent because it is irrelevant. Thus (\ref{stricteqcon1}) is exactly the same as (\ref{eqcon1}). In (\ref{stricteqcon2}), on the other hand, we hope to describe the change of policies when stopping part is from stopping to continuing. Here comes a dilemma: in this case deviation for control policy only matters in continuation region, while for stopping policy it is more important to consider stopping region. That is why we require all $(t,x)\in E$ rather than $(t,x)\in C$ to satisfy (\ref{stricteqcon2}). This paradox is the most important reason that makes us think that weak equilibrium is a better concept for studying time-inconsistent stopping control problems. This point is also the crucial point that makes strict weak equilibrium stronger than weak equilibrium. In Section \ref{eqchar} after stating the characterization theorem in infinite time horizon case (Theorem \ref{thminf}), we will briefly explain why strict weak equilibrium implies a weak one (see Remark \ref{strictchar}). At the end of Section \ref{eqchar} we provide an example showing a weak equilibrium need not to be strict (see Remark \ref{counterexm}). In the rest of this paper, we focus on weak equilibrium.

}
\section{Characterizations of the equilibrium}
\label{eqchar}
We now present the main results of this paper, including assumptions and technical lemmas that are needed for proofs. This section is further divided into two subsections. In Subsection \ref{eqcharfinite} we present results for finite time horizon, and all proofs are provided in Appendix \ref{proofsec3}. In Subsection \ref{eqcharinf} we briefly discuss the infinite time horizon case, which will be used for in Section \ref{exm}, but we omit all the proofs there because they are similar to the finite time horizon case.
\subsection{Finite time horizon}
\label{eqcharfinite}
In this subsection we fix $(\hat{u},C)\in \Ui\times \Oi$ and investigate whether it is equilibrium. For convenience, we drop all the superscripts $\hat{u}$, e.g., $\Lambda=\Lambda^{\hat{u}}$, $\Sigma=\Sigma^{\hat{u}}$, $X= X^{\hat{u}}$ if needed. Moreover, we denote by $L^{\infty}_{\rm poly}$ the space of function that has at most polynomial growth at infinity, i.e.,
\[
L^{\infty}_{\rm poly}=\{f:E\to \mR \big | |f(t,x)|\leq M(1+\|x\|)^{\gamma} {\rm\ as\ }\|x\|\to \infty, {\rm\ for\ some\ }\gamma>0{\rm \ and\ }M>0\}.
\]
We have the following main result of this paper:
\begin{theorem}
\label{mainthm}
If the {\bf auxiliary function} $f(s,y,t,x)\triangleq\mE^{t,x}g(s,y,\hat{\tau},X^{\hat{u}}_{\hat{\tau}})$ satisfies:
  \begin{align}
  \label{fcond}
  \tag{H1}
  &f(s,y,\cdot,\cdot)\in C^{1,2}(C)\cap C^0(E)\cap L^{\infty}_{\rm poly},\forall (s,y)\in E.
  \end{align}
  Then $(\hat{u},C)\in \Ui\times \Oi$ is an equilibrium if and only if $f$ and $C$ solve the following system:
  \begin{align}
     & A^{\hat{u}}f(s,y,t,x) = 0, \forall (t,x)\in C, (s,y)\in E, \label{sys1}\\
     & \sup_{\bu \in \bU}A^{\bu}f(t,x,t,x) = 0, \forall (t,x)\in C, \label{sys2}\\
     & A^{\hat{u}}g(t,x,t,x)\leq 0,\forall (t,x)\in \intt(D), \label{sysplus}\\
     & \limsup_{\epsilon \to 0}\frac{\mE^{t,x}[f(t,x,t+\epsilon,X_{t+\epsilon})-f(t,x,t,x)]}{\epsilon}\leq 0,\forall (t,x)\in \partial C, \label{sys3}\\
     & f(s,y,T,x) = g(s,y,T,x),\forall x,y\in X, s\in [0,T), \label{sys4}\\
     & f(s,y,t,x) = g(s,y,t,x), \forall (t,x)\in D, (s,y)\in E,\label{sys5}\\
     & f(t,x,t,x)\geq g(t,x,t,x), \forall  (t,x)\in E, \label{sys6}
  \end{align}
   where $D \triangleq E \backslash C$.

   Furthermore, when $(\hat{u},C)$ is equilibrium, we have
   \begin{equation}
\hat{u}(t,x)=\mathop{{\rm argmax}}_{\bu\in \bU}A^{\bu}f(t,x,t,x),\label{equiuhat}
   \end{equation}
   for those $(t,x)\in C$ such that the map $\displaystyle (t,x)\mapsto \mathop{{\rm argmax}}_{\bu\in \bU}A^{\bu}f(t,x,t,x)$ is well-defined.
\end{theorem}
Theorem \ref{mainthm} builds upon the following several lemmas, which calculate the limits (\ref{eqcon2}) and (\ref{eqcon3}). Their proofs are given in Appendix \ref{proofsec3}.

\begin{lemma}
\label{lm1}
For any $(t,x)\in E\backslash \partial C$,
\[
\limsup_{\epsilon \to 0}\frac{J(t,x;\hat{u},\tau_{(\hat{u},C,t+\epsilon)})-f(t,x,t,x)}{\epsilon}
=A^{\hat{u}}f(t,x,t,x).
\]
\end{lemma}
\begin{lemma}
\label{lm2}
For any $(t,x)\in C$, $\bu \in \bU$,
\[
\limsup_{\epsilon \to 0}\frac{J(t,x;\hat{u}_{(t,\epsilon,\bu)},\tau_{(\hat{u}_{(t,\epsilon,\bu)},C,t)})-f(t,x,t,x)}{\epsilon}=
A^{\bu}f(t,x,t,x).
\]
\end{lemma}
\begin{lemma}
\label{lm3}
For any $(\hat{u},C)\in \Ui\times \Oi$, (\ref{sys1}), (\ref{sys4}) and (\ref{sys5}) hold.
\end{lemma}

\begin{proof}[Proof of Theorem \ref{mainthm}]
  (\ref{sys1}), (\ref{sys4}) and (\ref{sys5}) hold no matter whether $(\hat{u},C)$ is equilibrium or not, thanks to Lemma \ref{lm3}. Clearly, by Lemmas \ref{lm1} and \ref{lm2}, (\ref{eqcon1}) is equivalent to (\ref{sys6}), (\ref{eqcon3}) is equivalent to (\ref{sys2}) with $\leq$, (\ref{eqcon2}) in $\intt(D)$ is equivalent to (\ref{sysplus}), (\ref{eqcon2}) on $\partial C$ is equivalent to (\ref{sys3}). Moreover, taking $(s,y)=(t,x)$ in (\ref{sys1}), (\ref{sys2}) with $\leq$ can be replaced by the one with $=$. (\ref{equiuhat}) is obvious from (\ref{sys1}) and (\ref{sys2}).
\end{proof}
Note that Theorem \ref{mainthm} provides full characterization of equilibrium if we know {\it a priori} the form of auxiliary function. However, using Theorem \ref{representation} below, we know that if $f$ solves some part of the system mentioned above, it is exactly the auxiliary function needed.
\begin{theorem}
  \label{representation}
  If $f$ satisfies (\ref{fcond}), and solves (\ref{sys1}), (\ref{sys4}) and (\ref{sys5}) for some $(\hat{u},C)\in \Ui\times \Oi$, then $f(s,y,t,x)=\mE^{t,x}g(s,y,\hat{\tau},X^{\hat{u}}_{\hat{\tau}})$.
\end{theorem}
\begin{proof}
  We only need to consider $(t,x)\in C$. For $n\geq 1$, define the stopping time { $\tau_n=\inf\{r\geq t: \|\sigma(r,X_r,\hat{u}(r,X_r))\nabla_x f(r,X_r)\|\geq n ,\mathrm{or\ }d((r,X_r),\partial C)<1/n\}\wedge T$}. Using It$\hato$'s formula, we have
  \[
  \mE^{t,x}f(s,y,\tau_n,X_{\tau_n})=f(s,y,t,x).
  \]
 Based on $f(s,y,\cdot,\cdot)\in L^{\infty}_{\rm poly}$ and dominated convergence theorem, letting $n\to \infty$, we get
 \begin{align*}
 f(s,y,t,x)&=\mE^{t,x}f(s,y,\hat{\tau},X_{\hat{\tau}})\\
           &=\mE^{t,x}f(s,y,\hat{\tau},X_{\hat{\tau}})I_{\{\hat{\tau}<T\}}+\mE^{t,x}f(s,y,T,X_T)I_{\{\hat{\tau}=T\}}\\
           &=\mE^{t,x}g(s,y,\hat{\tau},X^{\hat{u}}_{\hat{\tau}}).
 \end{align*}
 Here we have used (\ref{sys4}), (\ref{sys5}) and the fact $(\hat{\tau},X_{\hat{\tau}})\in D$, for $\mP^{t,x}$-almost surely $\omega \in \{ \hat{\tau}<T  \}$.
\end{proof}
In order to provide refined characterization of the equilibrium, and prepare for the verification procedure in Section \ref{exm} at the same time, we discuss the boundary condition (\ref{sys3}).
Similar to \citet*{Christensen2020} and many other literature where free boundary problems play important roles, one usually expects {\bf smooth fitting principle} to make an ansatz. We propose smooth fitting in multi-dimensional setting as follows:
\begin{equation}\label{spatialfitting}
	f_x(t,x,\cdot,\cdot)\big|_{(t,x)}=g_x(t,x,\cdot,\cdot)\big|_{(t,x)},\forall (t,x)\in \partial C.
\end{equation}
Indeed, we prove that, under mild conditions, this is {\it necessary} for equilibrium condition (\ref{sys3}). On the other hand, it is also helpful to the verification procedure if we can obtain some {\it sufficient} condition for (\ref{sys3}). In fact, we propose:
\begin{equation}\label{boundarysc}\tag{\ref{sys3}'}
  \limsup_{\substack{(s,y)\notin \partial C\\(s,y)\to (t,x)}}A^{\hat{u}}f(t,x,s,y)\leq 0,\forall (t,x)\in \partial C.
\end{equation}
 When making ansatz, we usually aim to find a solution that has continuous spatial derivatives, especially when spatial dimension $n=1$. Therefore, combined with (\ref{boundarysc}), (\ref{spatialfitting}) is also an appropriate sufficient condition. To provide connections between (\ref{spatialfitting}) and (\ref{sys3}), we need the following mild technical assumption on $f$:
\begin{equation}\label{fcon2} \tag{H2}
  { \left.f(s,y,\cdot,\cdot) \right|_C } {\rm\ extends\ to\ some\ } \tilde{f}(s,y,\cdot,\cdot)\in C^{1,2}({ E}),\forall (s,y)\in E.
\end{equation}
{  Although $f$ is already defined on $E$, it is not $C^{1,2}$ across the boundary $\partial C$. The key point of (\ref{fcon2}) is that we consider the restriction of $f$ on $C$ (which is $C^{1,2}$) and extend it {\it smoothly} to $E$. This is needed for the application of local time formula (see Appendix \ref{proofsec3}). Moreover, (\ref{spatialfitting}) is in fact not perfectly rigorous because $f$ is not smooth across $\partial C$, but now based on (\ref{fcon2}) we see that it actually means $\tilde{f}_x=g_x$ on the boundary.} We have the following theorem, which is another main result of this paper:
\begin{theorem}
\label{boundary}
  Suppose that (\ref{fcond}) and (\ref{fcon2}) hold, then
  \begin{itemize}
    \item[(1)](\ref{boundarysc})+(\ref{spatialfitting}) $\Longrightarrow$ (\ref{sys3}).
    \item[(2)]{ If $\partial C\in C^2$ and $\Lambda(t,x)^{\mt}\Lambda(t,x)$ is invertible for any $(t,x)\in \partial C$, (\ref{sys3})$\Longrightarrow$ (\ref{spatialfitting})}.
  \end{itemize}
\end{theorem}

\begin{remark}
\label{boundaryrmk}
  (\ref{boundarysc}) is a convenient sufficient condition for the proof of the previous theorem. In fact, if $(\hat{u},C)$ is an equilibrium, then before taking limit in (\ref{boundarysc}) it equals to 0 for $(s,y)\in C$, thanks to (\ref{sys1}). In $D$, we know $f=g$ and $g$ is something we know a priori. Thus, one sufficient condition to verify (\ref{boundarysc}) is $A^{\hat{u}}g(t,x,s,y)\leq 0$, $\forall (s,y)\in \intt(D), (t,x)\in {\partial C}$.
\end{remark}
\begin{remark}
  Recall that $\Lambda(t,x)$ is the diffusion coefficient under the candidate equilibrium strategy $\hat{u}$. The conclusion (2) in Theorem \ref{boundary} thus asserts that (\ref{spatialfitting}) is necessary if we only consider the equilibria that give nondegenerate diffusion term. This is true in many situations, including those where the diffusion term cannot be controlled.
\end{remark}
\begin{remark}
  Smooth fitting principles are always important topics in optimal stopping theory. For review of classical results on this topic, see \citet*{Goran2006}. Recently efforts have been made to establish global $C^1$ regularity of value function of optimal stopping problems, see \citet*{DEANGELIS2020} and \citet*{Cai2021} for examples. {  We have obtained similar results for time-inconsistent stopping control under additional regularity imposed on stopping boundary. It is very interesting to try to drop this assumption and {\it prove} the regularity directly from equilibrium conditions. We hope to establish this type of results in future works.}
\end{remark}
\begin{remark}
  We now make some comments on assumptions (\ref{fcond}) and (\ref{fcon2}). It is known that in control theory, to make regularity assumptions on value function is sometimes restrictive (see e.g., Example 2.3 on page 163 of \citet*{Yong1999}). However in our context, when fixing $(s,y)$ and $\hat{u}$, $f$ has clear connection to {\it linear} parabolic equations with initial and boundary value. Thus, PDE theory helps to establish regularity {\it inside} $C$. Moreover, the growth of $f$ itself follows directly from that of $g$, and we do not require the growth of derivatives of $f$, weakening the condition in \citet*{He2019}. At last, the assumption (\ref{fcon2}) can be established by uniform continuity and Whitney's extension theorem, thanks to the H$\ddo$lder estimates for parabolic equations. See Appendix \ref{fconddiscussion} for details.
\end{remark}
\subsection{Infinite time horizon}
\label{eqcharinf}
To prepare for the studying of the concrete examples in Section \ref{exm}, for which we hope to get analytical solutions, we extend the main results in the previous subsection to infinite time horizon case. One of the advantages for adapting the (weak) equilibrium formulation is that all the arguments are local (see all the proofs in Appendix \ref{proofsec3}). Therefore, it is almost trivial to exntend Theorem \ref{mainthm} once we neglect the condition (\ref{sys4}). Furthermore, Theorem \ref{boundary} still holds true for the same reason. The only thing we need to refine for the infinite time horizon theory is that the estimate (\ref{littleep2}) in Appendix \ref{proofsec3} seems not to be directly applicable anymore. However, the estimate there is still valid, except that we do not substitute $T=\infty$, but choose $T>t+1$ for any fixed $t$. This is sufficient because when applying it, we only consider the behaviour for trajectory of $X$ on $[t,t+\epsilon]$. The similar arguments in the proofs of results in Subsection \ref{eqcharfinite} (see Appendix \ref{proofsec3}) are still applicable and anything follows. We thus have the following:
\begin{theorem}
  \label{thminf}
  If the auxiliary function $f(s,y,t,x)\!\triangleq \!\mE^{t,x}g(s,y,\hat{\tau},X^{\hat{u}}_{\hat{\tau}})I_{\{\hat{\tau}<\infty\}}$ satisfies (\ref{fcond}), then $(\hat{u},C)\in \Ui\times \Oi$ is an equilibrium if and only if $f$ and $C$ solve the following system:
  \begin{align}
     & A^{\hat{u}}f(s,y,t,x) = 0, \forall (t,x)\in C, (s,y)\in E, \label{sysinf1}\\
     & \sup_{\bu \in \bU}A^{\bu}f(t,x,t,x) = 0, \forall (t,x)\in C, \label{sysinf2}\\
     & A^{\hat{u}}g(t,x,t,x)\leq 0,\forall (t,x)\in \intt(D), \label{sysinfplus}\\
     & \limsup_{\epsilon \to 0}\frac{\mE^{t,x}[f(t,x,t+\epsilon,X_{t+\epsilon})-f(t,x,t,x)]}{\epsilon}\leq 0,\forall (t,x)\in \partial C, \label{sysinf3}\\
     & f(s,y,t,x) = g(s,y,t,x), \forall (t,x)\in D, (s,y)\in E,\label{sysinf5}\\
     & f(t,x,t,x)\geq g(t,x,t,x), \forall (t,x)\in E. \label{sysinf6}
  \end{align}
  Furthermore, when $(\hat{u},C)$ is equilibrium, we have
   \begin{equation}
\hat{u}(t,x)=\mathop{{\rm argmax}}_{\bu\in \bU}A^{\bu}f(t,x,t,x),\label{equiuhatinf}
   \end{equation}
   for those $(t,x)\in C$ such that the map $\displaystyle (t,x)\mapsto \mathop{{\rm argmax}}_{\bu\in \bU}A^{\bu}f(t,x,t,x)$ is well-defined.
\end{theorem}
{ 
\begin{remark}\label{strictchar}
	We now make comments on the concept of strict weak equilibrium (see Definition \ref{stricteqdef}). From the proof of Theorem \ref{mainthm} (Appendix \ref{proofsec3}), the equilibrium condition (\ref{stricteqcon2}) is equivalent to
	\[
	\limsup_{\epsilon \to 0}\frac{\mE^{t,x}[f(t,x,t+\epsilon,X^\bu_{t+\epsilon})-f(t,x,t,x)]}{\epsilon}\leq 0,\forall\  (t,x)\in E,\bu \in \bU.
	\]
	The subtle point is that this requirement is imposed on {\it any} point $(t,x)$ in the {\it whole} state space and {any\it} admissible control $\bu$, not just the equilibrium one. Translating to the infinite time horizon case, the characterization system of strict weak equilibrium is similar to that in Theorem \ref{thminf}, with (\ref{sysinfplus}) revised to
	\begin{equation}\label{sysinfplusstric}\tag{\ref{sysinfplus}'}
	A^\bu g(t,x,t,x)\leq 0, \forall \bu \in \bU, (t,x)\in \intt(D).
	\end{equation}
	We immediately conclude that a strict weak equilibrium must be weak.
\end{remark}
\begin{remark} 
	In Theorem \ref{thminf}, we define the auxiliary function to be the expected reward restricted on the event $\{\hat{\tau}<\infty\}$. On the one hand, this event does not necessarily have probability 1. On the other hand, it is reasonable to assume that outside this event, i.e., in the ``never stop" scenario, the reward is zero, because it is never realized.
\end{remark}
In applications, especially in the infinite time case, the controlled $X$ is usually time-homogenous and $g$ takes the form $g(s,y,t,x)=\delta(t-s)g(x,y)$, with {\it discount function} $\delta$ satisfying $\delta(0)=0$ and $\delta(\cdot)$ nonincreasing. Under these additional assumptions, the system used to characterize the equilibrium solution becomes more elegant. Indeed, the first simplification is from the time symmetry: we can now take the open subsets $C$ of $\mX$ (under its relative topology) with $C^2$ boundary as admissible stopping polices, and taking $\tilde{C}=C\times [0,\infty)\in \Oi$ brings us to notations of the previous sections. Furthermore, it is straightforward to show that under this choice of stopping policy, $f$ has the form $f(s,y,t,x)=\mE^{x}\delta(\tau_C+t-s)g(X^{\hat{u}}_{\tau_C},y)I_{\{\tau_C<\infty\}}$, where $\tau_C=\tau_{(\hat{u},C\times[0,\infty),0)}$ under the notation of Section \ref{notation}. Therefore, restricting on the diagonal, functions $f$ and $g$ do not depend on $t$, and we denote $f^d(x)=f(t,x,t,x)$ and $g^d(x)=g(t,x,t,x)$. To ease our notation, we still denote by $\partial_x$ the spatial deferential with respect to the {\it second} $x$ variable only, i.e., $\partial_x f^d(x)\triangleq \partial_x f(t,x,t,\cdot)|_x$. Combining the above discussions, we have the following corollary, which will be used repeatedly in concrete examples in Section \ref{exm}:
\begin{corollary}
  \label{habitcoro}
   Suppose that the dynamic of diffusion $X$ under admissible control is time-homogeneous, {  that $\Lambda^{\hat{u}}(x)^{\mt}\Lambda^{\hat{u}}(x)$ is invertible for any $x\in \partial C$, and} that the reward function has the form $\delta(t-s)g(x,y)$. Moreover, assume that under the strategy pair $(\hat{u},\tau_C)$, the auxiliary function $f(s,y,t,x)=\mE^x\delta(\tau_C+t-s)g(X^{\hat{u}}_{\tau_C},y)I_{\{\tau_C<\infty\}}$ satisfies (\ref{fcond}) and (\ref{fcon2}). Then $(\hat{u},C\times [0,\infty))$ is an equilibrium if and only if $f$ and $C$ solve the following system:
  \begin{align}
     & A^{\hat{u}}f(s,t,x,y) = 0, \forall x\in C, y\in \mX, s,t\geq 0,\label{sysinfg1}\\
     & \sup_{\bu \in \bU}A^{\bu}f(t,x,t,x) = 0, \forall x\in C,t\geq 0, \label{sysinfg2}\\
     & A^{\hat{u}}g(t,x,t,x)\leq 0,\forall x\in \mX \backslash C,t\geq 0, \label{sysinfgplus}\\
     & \partial_xf^d(x)=\partial_xg^d(x),\forall x\in \partial C, \label{sysinfg3}\\
     & f(s,y,t,x) = g(s,y,t,x), \forall x\in \mX\backslash C, y\in \mX,s,t\geq 0,\label{sysinfg5}\\
     & f^d(x)\geq g^d(x), \forall x\in \mX. \label{sysinfg6}
  \end{align}
\end{corollary}
\begin{remark}\label{counterexm}
	After providing smooth fitting result in infinite time horizon case, we are able to provide an { one-dimensional} example where a weak equilibrium need not to be strict. Consider { a discount function $\delta$ with $-\infty<\delta'(0)<0$ and $g(s,y,t,x)=\delta(t-s)x^3$. Take admissible control policy as $\Ui=\{\theta_0,\theta_1\}$ (only two constant controls allowed), where $\theta_0=0$ and $\theta_1>0$. We assume that the state process under control $\theta$ is $X^\theta_t=\theta W_t$, where $W_t$ is a Brownian motion. Take the continuation region $C=(-\infty,0)$ and control $\theta=\theta_0$. Note that for $x<0$, $\mP^x(\tau_C=\infty)=1$. Therefore, $f(s,t,x)\equiv 0$ in $C$. Moreover, $\partial_x g^d(0)=0=\partial_x f^d(0)$, $f^d \geq g^d=x^3$ if $x\leq 0$, and $A^{\theta_0}g(t,x,t,x)=\delta'(0)x^3+3\theta_0 x<0$ for $x>0$. Thus, based on Corollary \ref{habitcoro}, it is trivial that $(\theta_0,C)$ is a weak equilibrium\footnote{Note that the sufficiency part of Corollary \ref{habitcoro} does not need the non-degeneracy condition of $\Lambda^{\hat{u}}$. See Theorems \ref{boundary} and \ref{thminf}. Therefore, it is valid to show that $\theta_0$ consists a weak equilibrium even it does not satisfy the non-degeneracy condition.}. However, $A^{\theta_1}g(t,x,t,x)=\delta'(0)x^3+3\theta_1 x>0$ for $0<x<\sqrt{3\theta_1/(-\delta'(0))}$, violating (\ref{sysinfplusstric}), proving that $(\theta_0,C)$ is not strict. Choosing $\delta(t)=e^{-\beta t}$ for some $\beta>0$, the problem is time-consistent. In this case, we provide an example where a weak equilibrium strategy need not to be optimal. It is interesting to investigate the relations between weak equilibrium and strict weak equilibrium (say, to find conditions under which a weak equilibrium is strict, or to investigate whether a strict equilibrium is optimal for time-consistent problems). This direction is left for future studies.}
\end{remark}
}
\section{Examples}
\label{exm}
In this section we investigate { three} concrete examples to illustrate the usage of the theoretical results established in the last section. In Subsection \ref{exmmodel} we introduce the general formulation of an investment-withdrawal decision problem, and in Subsections \ref{example1} and \ref{nonexp} we study the examples in detail. { In Subsection \ref{2dexm} we provide an example in 2-dimension, where the equilibrium stopping boundary is a circle}.
\subsection{Investment-withdrawal decision model}
\label{exmmodel}
For simplicity we assume that the decision maker has the opportunity to invest in one stock:
\[
\md S_t/S_t=\mu \md t+\sigma \md W_t,
\]
where $\mu$ and $\sigma$ are both positive constants. In this situation, it is usually assumed $\mu\neq 0$. We only consider the case $\mu>0$. Suppose that the decision maker can invest on $(0,\infty)$. The decision maker will also establish (for himself) a withdrawal mechanics. For example, he will choose a {\it continuation region} $C$. Once his wealth $X_s$ leaves $C$, he will withdraw all his investment and stop any exchange for the moment. Therefore, choosing (time-homogenous) investment proportion $\theta(\cdot)$ and continuation region $C$, the wealth dynamics and the expected pay off are expressed respectively by
\begin{align}
  & \md X^{\theta}_t=\mu\theta(X^{\theta}_t)X_t\md t+\sigma \theta(X^{\theta}_t)X_t \md W_t \label{wealthdynamic},\\
  & J(t,x;\theta,\tau)=\mE^{t,x}g(t,x,\tau,X^{\theta}_\tau)I_{\{\tau<\infty\}}.\label{valuefun}
\end{align}

When making decisions, the decision maker takes $\tau=\tau_{(\theta,C\times [0,\infty),0)}$. The reward function $g$ will be specified in the following several subsections.
\subsection{Reduction of utility by wealth level}
\label{example1}
\citet*{Christensen2018} studies an example of time-inconsistent stopping problem, where the reward is affected by the current wealth level $x$. Using the investment-withdrawal model developed in the last subsection, that example can be interpreted as an investment chance where the proportion of money invested is locked (or set) to be 1 and the agent can choose a time to withdraw. Here we extend this example to the occasion where the agent is allowed to choose and adjust the investment in response to the market performance and can also decide when to withdraw the investment. Now we briefly introduce the setting. To obtain semi-analytical solution we shall work on the case $T=\infty$. If the agent, with his wealth $x$, chooses the investment strategy $\theta$ (which depends on $x$ only) and decides to stop at $\tau$, then he will get
\begin{equation}
\label{exmpro1}
\mE^x e^{-\beta \tau}\left\{1-\exp\{-a[X^{\theta}_{\tau}-h(x)-k]\}\right\}I_{\{\tau<\infty\}} ,
\end{equation}
where $a$, $\beta$ and $k$ are positive constants, and $h$ is an increasing function with $h(0)=0$. It is seen that the larger $x$, the less he will actually get at $\tau$, given the same outcome of $X^{\theta}_{\tau}$. This is well-interpreted economically: to get ten thousand dollars means a lot to the homeless, but could mean nothing to a billionaire. We also require $x-h(x)$ to be nondecreasing, which means once withdrawn, more wealth gives more utility.
\vskip 5pt
Under the setting mentioned above, using the notations in Corollary \ref{habitcoro} we have $n=1$, $\mX=(0,\infty)$, $g(s,t,x,y)=e^{-\beta(t-s)}[1-\exp\{-a[x-h(y)-k]\}]$, $\bU=(0,\infty)$, and, to find equilibrium, we try to find $f$ and $C$ solving (\ref{sysinfg1})-(\ref{sysinfg6}).
{  In this subsection we aim to find equilibrium among those with a threshold type stopping strategy, i.e., $C=(0,x^*)$ for some $x^*>0$\footnote{ Using (\ref{sysinfgplus}), if $(\theta,C\times [0,\infty))$ is an equilibrium, it can be shown that there exists $\underline{x}>0$ such that $(0,\underline{x})\subset C$ (see similar arguments in the proof of Proposition \ref{exmirrational}). We assume $C=(0,x^*)$ for simplicity.}.
	
As a first step, we first use the necessity part of Corollary \ref{habitcoro} to narrow down our search to a candidate solution. To do this, suppose $(\theta, C\times[0,\infty))$ is an equilibrium, where $x\mapsto x\theta(x)$ is Lipschitz, $\theta(0+)<\infty$, and $C=(0,x^*)$. As a well known result from one-dimensional diffusion theory, $0$ is an inaccessible boundary point for $X$ (see \citet{ito1996diffusion} and \citet{Helland1996}), thus $\tau_C=\tau^\theta_{x^*}\triangleq\inf\{t\geq 0:X^\theta_t=x^*\}$, $\mP^x$-a.s. for any $0<x<x^*$. Therefore we have
\[
f(s,y,t,x)=e^{-\beta(t-s)}g(x^*,y)\phi(x),
\]
where $\phi$ solves
\begin{equation}
	\label{eqne}
\left\{
\begin{aligned}
	&\frac{1}{2}\sigma^2\theta(x)^2x^2\phi''(x)+\mu\theta(x)\phi'(x)=\beta\phi(x),\\
	&\phi(0)=0,\phi(x^*)=1.
\end{aligned}
\right.
\end{equation}
Using Corollary \ref{habitcoro}, we know
\begin{equation}\label{theta}
\theta(x)=-\frac{\mu f_x\big|_{(x,x)}}{\sigma^2 x f_{xx} \big| _{(x,x)}}=-\frac{\mu\phi'(x)}{\sigma^2x \phi''(x)}.
\end{equation}
Plugging this back into (\ref{eqne}), we conclude that $\phi$ must satisfy
\begin{equation}
	\label{eqne2}
	\left\{
	\begin{aligned}
		&\frac{1}{2}\kappa\frac{(\phi')^2}{\phi''}+\beta \phi=0,\\
		&\phi(0)=0,\phi(x^*)=1,
	\end{aligned}
	\right.
\end{equation}
with $\kappa=\mu^2/\sigma^2$. Moreover, $\phi$ is strictly increasing in $(0,x^*)$. There are multiple ways tho show that (\ref{eqne2}) has only one increasing solution $\phi(x)=(x/x^*)^\alpha$, $\alpha=2\beta/(2\beta+\kappa)$\footnote{ For example, one can transform the equation of (\ref{eqne2}) into $\frac{1}{2}\kappa(\log \phi')'+\beta(\log \phi)'=0$ and integrating backwards from $x^*$. Details can be provided upon requirements. One can also use inverse function method similar to the proof of Proposition \ref{appDpro} in Appendix \ref{extsBVP}.}. Using (\ref{theta}), we know $\theta(x)\equiv\theta^*\triangleq \frac{2\beta }{\mu}+\frac{\mu}{\sigma^2}$. Therefore we have proved the following result:
\begin{proposition}
	Suppose $(\theta,(0,x^*)\times [0,\infty))$ is equilibrium, then $\theta(x)\equiv \theta^*$.
\end{proposition}

The next step is to determine a threshold $x^*$ and provide sufficient conditions to ensure that our candidate strategy is indeed equilibrium.}
To this end, we note that $f$ has the following form inside $C$:
\[
f(s,t,x,y)=e^{-\beta(t-s)}\left( \frac{x}{x^*}  \right)^{\alpha(\theta^*)}(1-\exp\{-a(x^*-h(y)-k) \}),
\]
where $\alpha(\theta)=\frac{1}{2}-\frac{\mu}{\sigma^2 \theta}+\sqrt{\frac{2\beta}{\sigma^2\theta^2}+\left(\frac{1}{2}-\frac{\mu}{\sigma^2 \theta}  \right)^2}$. Here we introduce the function $\alpha(\cdot)$ for convenience of numerical experiments, where for comparison, the control may be locked to other constants rather than $\theta^*$. Clearly $f$ and $C$ satisfy (\ref{sysinfg1}) and (\ref{sysinfg5}). To show that (\ref{sysinfg3}) holds, we need
\[
\partial_x f^d(x^*)=\partial_xg^d(x^*),
\]
which, after direct computation, becomes
\begin{equation}\label{exmsmoothfit}
  \alpha(\theta^*)=[ax^*+\alpha(\theta^*)]\exp\{-a[x^*-h(x^*)-k]\}.
\end{equation}
Under the choice $(\theta^*,x^*)$ described by (\ref{exmsmoothfit}), (\ref{sysinfgplus}) becomes, for any $x\geq x^*$,
\begin{equation}
\label{exmeqcond1}
\begin{aligned}
-\beta g&+\mu \theta^* xg_x+\frac{1}{2}\sigma^2 (\theta^*)^2x^2 g_{xx}\\
&=-\beta+\left[\beta+\mu\theta^*ax-\frac{1}{2}\sigma^2(\theta^*)^2a^2 x^2\right]\exp\big\{-a[x-h(x)-k]\big\}\leq 0,
\end{aligned}
\end{equation}
We consider the following two cases of $x$:
\begin{itemize}
  \item[{\bf Case1.}]$\beta+\theta^*ax-\frac{1}{2}\sigma^2(\theta^*)^2a^2x^2\leq 0$.
  In this case it is obvious that (\ref{exmeqcond1}) is true.
  \item[{\bf Case2.}]$\beta+\theta^*ax-\frac{1}{2}\sigma^2(\theta^*)^2a^2x^2>0$.
Using the fact that $x-h(x)$ is increasing, we know
\[
\exp\big\{  -a(x-h(x)-k)\leq \exp\big\{  -a(x^*-h(x^*)-k)       \big\},\forall x\geq x^*.
\]
On the other hand,
\[
\exp\big\{  -a(x-h(x^*)-k)\leq \exp\big\{  -a(x^*-h(x^*)-k)       \big\},\forall x\geq x^*.
\]
Therefore, in this case, (\ref{exmeqcond1}) is true if
\begin{equation}
\label{666}
-\beta+\left[\beta+\mu\theta^*ax-\frac{1}{2}\sigma^2(\theta^*)^2a^2 x^2\right]\exp\left\{-a[x^*-h(x^*)-k]\right\}\leq 0.
\end{equation}
Using $\beta+\mu\theta ax-\frac{1}{2}\sigma^2 \theta^2a^2x^2 \leq \beta+\frac{\mu^2}{\sigma^2} $ and (\ref{exmsmoothfit}), we find that (\ref{666}) is true if
\[
-\beta+\left[\beta+\frac{\mu^2}{2\sigma^2}\right]\frac{\alpha(\theta^*)}{ax^*+\alpha(\theta^*)}\leq 0,
\]
which is equivalent to
\[
x^*\geq \underline{x}(\theta^*)
\]
with
\[
\underline{x}(\theta)\triangleq\frac{\mu^2 \alpha(\theta)}{2\beta\sigma^2 a }.
\]
\end{itemize}
To conclude, we have shown that (\ref{sysinfgplus}) holds if
\begin{equation}
\label{xstarcond}
x^* \geq  \underline{x}(\theta^*).
\end{equation}

Now the only condition left to be verified is (\ref{sysinfg6}), being equivalent to
\begin{equation}
\left( \frac{x}{x^*} \right)^{\alpha}\left[1-\exp\{-a[x^*-h(x)-k]\}\right]\geq 1-\exp\left\{-a[x-h(x)-k]\right\}, \forall x<x^*.
\label{fgeqg}
\end{equation}
This needs some further assumptions, which we will provide in Proposition \ref{exmconclusion} below\footnote{There are in fact gaps in the argument proving similar relations as (\ref{sysinfg6}) in \citet*{Christensen2018}. In pages 31-32 of that paper, the authors reduce their arguments to the case $h(x)=0$. However this reduction is not allowed directly. Indeed, from (\ref{exmsmoothfit}) we know that the value of $x^*$ depends crucially on the choice of $h$. Consequently, $x^*$ in the inequality (\ref{fgeqg}) are different when $h=0$. This leads to essential difficulties for proving (\ref{fgeqg}). We also provide counter example where (\ref{fgeqg}) is not true, even with simple choice of $h$ (see Remark \ref{commentsh} and Figure \ref{fig4}).}.
\begin{proposition}
\label{exmconclusion}
Let $\theta^* = \frac{2\beta}{\mu}+\frac{\mu}{\sigma^2}$, $\alpha=\frac{2\beta}{2\beta+\mu^2/\sigma^2}$. Denote by $x^*$ the solution of (\ref{exmsmoothfit}), and $x_0^*$ the solution of (\ref{exmsmoothfit}) with $h\equiv 0$, respectively. Suppose that the model parameters satisfy the following requirements:
 \begin{equation}
x_0^*> \frac{2-\alpha}{a}, \label{x0assumption}
\end{equation}
\begin{equation}
h''(x)\leq 0, \forall\ 0<x<x^*,\label{hassumption1}
\end{equation}
\begin{equation}
h'(x) <\min\left\{\frac{1}{2(1+\frac{1}{ax_0^* + \alpha-1})},\frac{1}{1-e^{-ax^*}}\cdot \left[   1-\frac{\alpha}{\min\{\alpha e^{ak},2e^{\alpha-2+ak}   \}}\right] ,\frac{1}{2}      \right\},\forall\ 0<x<x^*. \label{hassumption2}
\end{equation}
Then (\ref{fgeqg}) is satisfied. As a consequence, if (\ref{xstarcond}) is also true, $(\theta^*,(0,x^*)\times [0,\infty))$ is an equilibrium solution.
\end{proposition}
\begin{proof}
  See Appendix \ref{proofexmcon}.
\end{proof}
\begin{remark}
  The assumptions in Proposition \ref{exmconclusion}, as well as (\ref{xstarcond}), seem to be complicated. In fact, there are financial insights: on the one hand, if a constant holding ratio is rational, the withdrawal threshold (desired asset level) cannot be too low (see (\ref{xstarcond})), which will be revealed again in the next subsection under a different preference model; on the other hand, if the effect of habit on preference is too strong ($h'$ is large), we cannot explicitly determine the rational strategies. Indeed, a preference that is easily influenced by habit can be regarded as a modelling of irrationality. In the numerical experiments below, we will show that under a broad and reasonable set of market parameters, all these assumptions are satisfied.
\end{remark}
\begin{remark}
\label{commentsh}
  Here we emphasize that the assumptions on $h$ in Proposition \ref{exmconclusion} are just convenient sufficient conditions. There are indeed some other possible choice of $h$ that satisfies $f\geq g$, and there are also possible choice of $h$ that violates it. In Figure \ref{fig4} we simply try linear function $h$. It is seen that for $h(x)=0.4x$, although the condition proposed in Proposition \ref{exmconclusion} is invalid, $(\theta^*,x^*)$ described above is still an equilibruim. However, for $h(x)=0.6x$, clearly, $f<g$ for some $x<x^*$. Therefore, (\ref{sysinfg6}) is invalid, and $(\theta^*,x^*)$ is not an equilibrium. This reveals an interesting fact: smooth fitting principle (by which $x^*$ is determined) does not necessarily give an equilibrium solution. This phenomenon is studied in detail by \citet*{Tan2021}, { and in \citet*{Bodnariu2022}, it is shown that introducing a kind of local time pushed mixed stopping rules can fill this gap}. 
\end{remark}
\begin{figure}[!h]
\centering
\includegraphics[width=0.7\textwidth]{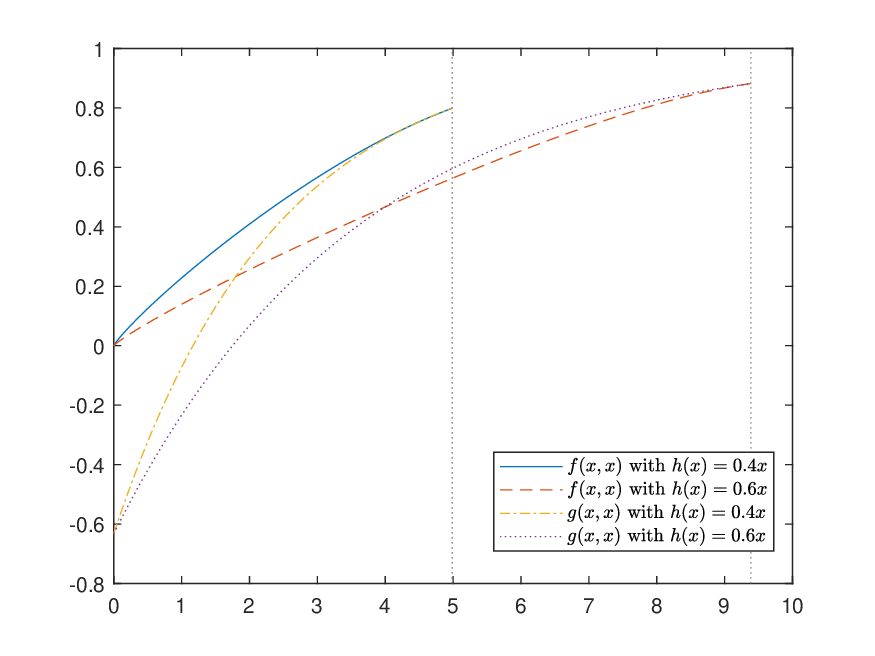}
\caption{Other choices of habit dependence function $h$.}
\label{fig4}
\end{figure}

We now take a look at some numerical examples. Here we take $\mu = 0.05$, $\sigma=0.3$, $\beta=0.1$, $k=0.7$, $a=0.7$. In this case, we choose $h(x)=0.15x$, then all assumptions in Proposition \ref{exmconclusion} are satisfied. Under this setting we have the equilibrium $\theta^*\approx 4.5556$, $x^*\approx 2.7919$. For comparison, if there is no habit dependence, i.e., $h=0$, then $\theta^*$ remains the same and the equilibrium withdrawal level is $x^*_0\approx 2.1090$. In this case, the problem is denegerated to time consistent one, and this is also the optimal investment and withdraw strategy (see \citet*{Karatzas2000}). If the investment level is locked to be $1$ and the investor is only allowed to choose the withdrawal time, then the corresponding equilibrium withdrawal level is $x^*_1\approx 1.9436$. It is seen that under the current market parameter, for sophisticated agent, the chance of discretionarily choose the investment level will make him improve the investment level while set a higher expectation wealth level. For the graphical illustrations of these results, see Figures \ref{fig2} and \ref{fig3}.
\begin{figure}[!h]
\centering
\includegraphics[width=0.7\textwidth]{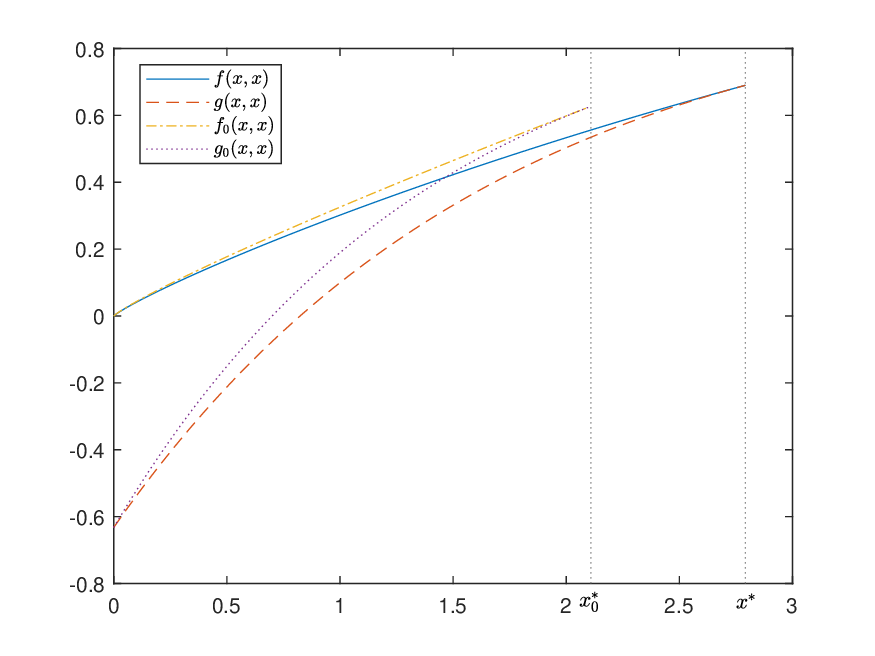}
\caption{$f_0$ is the equilibrium value for $h=0$, and $g_0$ is the corresponding reward function.}
\label{fig2}
\end{figure}
\begin{figure}[!h]
\centering
\includegraphics[width=0.7\textwidth]{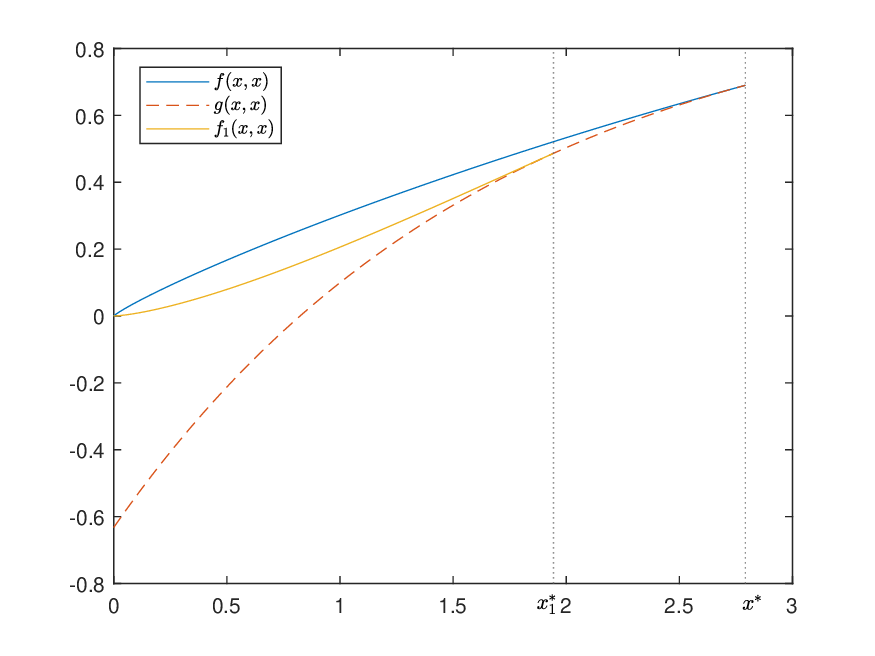}
\caption{$f_1$ is the equilibrium value when investment is locked to be 1. The reward function is still $g$.}
\label{fig3}
\end{figure}

Through numerical experiment, there are some novel financial insights from our investment-withdrawal decision model, comparing to stopping model without the discretionary investment opportunity. Specifically, there are different behaviors of the equilibrium withdrawal threshold $x^*$ when the volatility $\sigma$ change. Under non-exponential discount model, it is found in \citet*{Ebert2020} that the equilibrium withdrawal threshold $x^*$ increases with both $\mu$ and $\sigma$. We confirm this result under the endogenous habit formation model, as originally developed in \citet*{Christensen2018}. Using equilibrium theory developed in the present paper, we find that if the agent is provided with discretionary investment opportunity, the withdrawal threshold {\it decreases} with the volatility and still increases with return rate (see Figure \ref{fig666}). Because higher volatility implies more risk, and withdrawal threshold can be seen as the expectation of agent, this result is much more intuitive: giving other things the same, people should reduce their expectation when market risk becomes higher.
\begin{figure}[!h]
\centering
\includegraphics[width=0.7\textwidth]{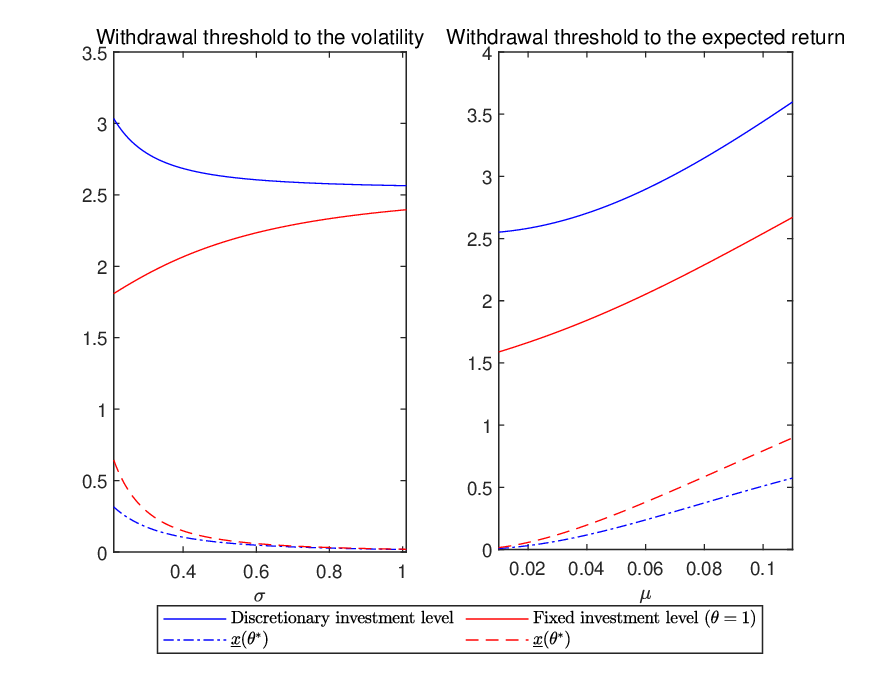}
\caption{The impact of volatility $\sigma$ and return rate $\mu$, on the equilibrium withdrawal threshold.}
\label{fig666}
\end{figure}
To assure that when $\mu$ and $\sigma$ are varying in the given range, the derived pair $(\theta^*,x^*)$ remains to be the equilibrium, we need to check the assumptions (\ref{x0assumption}) and (\ref{hassumption2}). For simplicity we denote by $M(\theta)$ the right hand side of (\ref{hassumption2}), where the dependence of $\theta$ comes from $\alpha=\alpha(\theta)$, $x^*=x^*(\theta)$ and $x^*_0=x^*_0(\theta)$. Figure \ref{fig777} justifies our analysis, and also shows that the assumptions we propose are reasonable.
\begin{figure}[!h]
\centering
\includegraphics[width=0.7\textwidth]{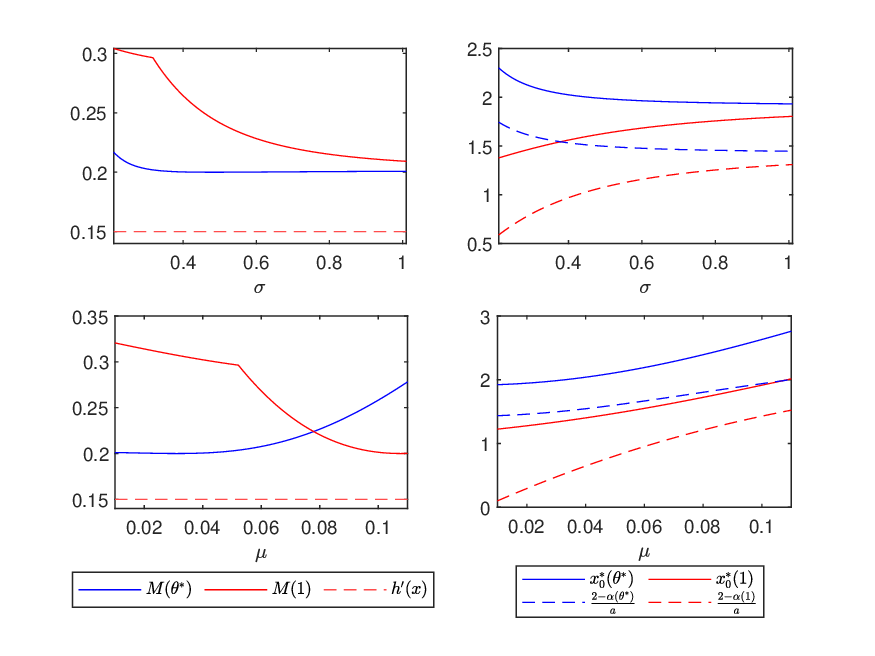}
\caption{Testing the assumptions on model parameters. Left column: testing (\ref{hassumption2}); Right column: testing (\ref{x0assumption}).}
\label{fig777}
\end{figure}
\subsection{Ambiguity on discount factor}
\label{nonexp}
In this subsection, we develop a decision model where the agent is uncertain (ambiguous) about his discount factor $\beta$, but has a belief on it. Under this setting, the problem he faces is time-inconsistent. By showing the nonexistence of {\it constant equilibrium} (the equilibrium solution with constant investment proportion), we argue that the theory proposed in this paper generates nontrivial results and is demanding for better understanding of the time-inconsistency in mathematical finance.

We consider the same model of underlying asset as in Subsection \ref{exmmodel}, while consider the following reward:
\begin{equation}
\label{payoff2}
g(s,t,x,y)=g(s,t,x)=\mB(t-s)\log(x),
\end{equation}
where $\mB$ is the {\it mean discount function} with belief $p$:
\begin{equation}
\label{discount}
\mB(t) = \int_0^{\infty}e^{-\beta t}p(\beta) \md \beta.
\end{equation}
Here $p$ is a probability density on the interval $(0,\infty)$. Recall that choosing continuation region $C$ and investment proportion $\theta$, the agent will implement $\tau=\tau_{(\theta,C\times [0,\infty),0)}$ (see the last paragraph in Subsection \ref{exmmodel}). For simplicity, we will use $\tau_{C}$ instead of $\tau_{(\theta,C\times [0,\infty],0)}$, if there is no confusion.
\begin{remark}
  The model is a time-inconsistent generalization of the example in Appendix A of \citet*{Karatzas2000}. We prove that there is no constant serving as equilibrium investment proportion (in other words, we exclude many {\it irrational} strategies). Because the uniqueness of equilibrium is hard to achieve, the fact that we can rule out many irrational strategies is quite insightful. To the best of our knowledge, this is the only negative result on the existence of constant equilibrium, except a recent paper \citet*{He2020} for a completely different problem (see (3) of Theorem 1 therein). In most of the existing literatures, where analytical solutions are attainable, the resulting equilibrium investment proportions are constant (or at least independent from the wealth). See \citet*{Karatzas2000} for time-consistent stopping control problems, \citet*{Ekeland2008}, \citet*{Yong2012}, \citet*{Bjork2017}, \citet*{Alia2017}, \citet*{He2019}, among others, for time-inconsistent control problems. Combining these observations we argue that the proposed model in the present paper, where the agent are supposed to make stopping {\it and} control decision {\it simultaneously} with the presence of time-inconsistency, brings essential differences from various existing models, and generates non-trivial results.
\end{remark}
\begin{remark}
	\label{discountexm}
  In this model the mean discount functions are special forms of general discount functions displayed in Subsection \ref{eqcharinf}. Indeed, as shown in \citet*{Ebert2020}, many popular discount functions can be expressed via a weighted distribution, similar to (\ref{discount}). Typical examples for our mean discount functions include:
  \end{remark}
  \begin{itemize}
    \item[(1)] ({\bf Quasi-exponential discount})
    \[\mB(t)=\lambda e^{-\beta_1 t}+(1-\lambda) e^{-\beta_2 t};\ \  p(\beta)=\lambda \delta_{\beta_1} + (1-\lambda) \delta_{\beta_2}, \lambda \in (0,1).\]
    \item[(2)] ({\bf Generalized hyperbolic discount})
    \begin{equation}\label{hyperbolic}
    \mB(t)=\frac{1}{(1+a t)^{b/a}};\ \  p(\beta)= \frac{\beta^{\frac{b}{a}-1}e^{-\frac{\beta}{a}}}{a^{\frac{b}{a}}\Gamma(\frac{b}{a})}, a>0,b>0.
    \end{equation}
    \item[(3)] ({\bf Compactly supported belief})
    \[\mB(t)=\int_{\ubeta}^{\obeta}e^{-\beta t} p(\beta)\md \beta,\ \ 0<\ubeta<\obeta<\infty\].
  \end{itemize}

Surprisingly, considering {\it real} ambiguity (i.e., $p$ is not singleton supported) will {\it exclude} constant investment proportion. Because the developed model is used to be describe rationality when optimal principle is not applicable, we have shown that under the developed model, any constant investment proportion is {\it irrational}. This conclusion (see Proposition \ref{exmirrational}) provides a possible explanation to many empirical documents (see e.g., \citet*{Wachter2010}) in contrary to classical Merton's suggestion or solution obtained in \citet{Karatzas2000} under time-consistent setting, which is constant proportion. Formally, we have the following:
\begin{proposition}
\label{exmirrational}
  If the reward function takes the form as in (\ref{payoff2}), and ${\rm supp}(p)$ is not singleton, then for any open subset $C$ of $(0,\infty)$ such that $C\neq (0,\infty)$, $C\neq \varnothing$, and any $\theta^*>0$, $(\theta^*,C\times [0,\infty))$ can not be equilibrium.
\end{proposition}
\begin{proof}
We show the proof by contradiction. Suppose that $(\theta^*,C\times [0,\infty))$ is indeed an equilibrium, we claim that ${\rm supp}(p)$ must be singleton. First, for any $x\in C$, we have $f_{xx}(t,t,x)<0$. Or otherwise suppose that there is $x\in C$ such that $f_{xx}(t,t,x)\geq 0$. Observing the following expression:
\[
A^{\theta} f(t,t,x) = f_t(t,t,x)+\mu\theta f_x(t,t,x)+\frac{1}{2}\sigma^2 \theta^2f_{xx}(t,t,x),
\]
we find that, choosing $\theta$ large enough (can be dependent on $(t,x)$), we are able to assure that $A^{\theta} f(t,t,x)>0$, contradicting (\ref{sysinf2}). Therefore, we assume that $f_{xx}(t,t,x)<0$. Because $C$ is open subset of $(0,\infty)$, it can be uniquely expressed by $C=\cup_{n=1}^{\infty} I_n$, where the family $\{I_n\}_{n=1}^{\infty}$ are countably many disjoint open intervals. We first claim that there must be one interval, say $I_1$, with the form $I_1=(0,r)$. Otherwise, we deduce that we can choose a sequence $x_k\to 0$ and $x_k\in \intt(D)$. However, using (\ref{sysinfplus}), we have
\[
-m_p \log x +\mu \theta^*-\frac{1}{2}\sigma^2(\theta^*)^2\leq 0, \forall x\in \intt(D),
\]
where
\[
m_p = \int_0^{\infty} \beta p(\beta)\md \beta>0.
\]
This leads to
\[
x \geq \exp(\frac{\mu \theta^*-\frac{1}{2}\sigma^2(\theta^*)^2}{m_p})>0,\forall x\in \intt(D),
\]
which clearly contradicts the fact that $x_k\to 0$ with $x_k\in \intt(D)$. In what follows, we focus on $I_1=(0,r)$.
  For any $x\in (0,r)$, $f(s,t,x)$ can be calculated by:
  \begin{align*}
f(s,t,x) &=\mE^x \mB(\tau_C-(s-t))\log(X^{\theta}(\tau_C))I_{\{\tau_C<\infty\}} \\
          &=\mE^x \mB(\tau_{(0,r)}-(s-t))\log(X^{\theta}(\tau_{(0,r)}))I_{\{\tau_{(0,r)}<\infty\}} \\
         &=\mE^x \log(X^{\theta}(\tau_{(0,r)}))\left[\int_0^{\infty} e^{-\beta (\tau_{(0,r)}-(s-t))} p(\beta)\md \beta      \right] \\
         &=\log r \left[\int_0^{\infty}e^{-\beta(t-s)}\left(  \frac{x}{r}\right)^{\alpha_+(\theta^*,\beta)}p(\beta)\md \beta\right],
\end{align*}
where $\alpha_+(\theta,\beta)=\frac{1}{2}-\frac{\mu}{\sigma^2 \theta}+\sqrt{\frac{2\beta}{\sigma^2\theta^2}+\left(\frac{1}{2}-\frac{\mu}{\sigma^2 \theta}  \right)^2}>0$. Using (\ref{sysinf3}) and Theorem \ref{boundary} (i.e., the weak smooth fitting principle), we have $f_x(t,t,r)=g_x(t,t,r)$, which yields
\[
\log r= 1\left/\left(\int_0^{\infty}\alpha_+(\theta^*,\beta)p(\beta) \md \beta \right)>0   \right.
\]
Direct calculation also shows
\[
xf_{xx}(t,t,x)=\log r\int_0^{\infty} \alpha_+(\theta^*,\beta)(\alpha_+(\theta^*,\beta)-1)\frac{x^{\alpha_+(\theta^*,\beta)-1}}{(r)^{\alpha_+(\theta^*,\beta)}}p(\beta)\md \beta.
\]
In light of (\ref{equiuhatinf}), for $0<x<r$, we consider
\begin{align*}
\tilde{\theta}(x)&= \mathrm{argmax}_{\theta\in \mR}A^{\theta}f(t,t,x)\\
&=-\frac{\mu f_x \big|_{(t,t,x)}}{\sigma^2 x f_{xx}\big |_{(t,t,x)}} \\
&= \frac{\mu}{\sigma^2}\cdot \frac{\int_0^{\infty} \alpha_+(\theta^*,\beta)\frac{x^{\alpha_+(\theta^*,\beta)-1}}{(r)^{\alpha_+(\theta^*,\beta)}}p(\beta)\md \beta}{\int_0^{\infty} \alpha_+(\theta^*,\beta)(1-\alpha_+(\theta^*,\beta))\frac{x^{\alpha_+(\theta^*,\beta)-1}}{(r)^{\alpha_+(\theta^*,\beta)}}p(\beta)\md \beta}
\end{align*}
We have $\tilde{\theta}(x)\equiv \theta^*$ for any $x\in (0,r)$. Taking derivative gives
\begin{align*}
\left[\int_0^{\infty} \alpha_+(\theta^*,\beta)\frac{x^{\alpha_+(\theta^*,\beta)-1}}{r^{\alpha_+(\theta^*,\beta)}}p(\beta)\md \beta   \right] &\left[ \int_0^{\infty} \alpha_+(\theta^*,\beta)(1-\alpha_+(\theta^*,\beta))^2\frac{x^{\alpha_+(\theta^*,\beta)-1}}{r^{\alpha_+(\theta^*,\beta)}}p(\beta)\md \beta    \right]\\
&-\left[ \int_0^{\infty} \alpha_+(\theta^*,\beta)(1-\alpha_+(\theta^*,\beta))\frac{x^{\alpha_+(\theta^*,\beta)-1}}{(r)^{\alpha_+(\theta^*,\beta)}}p(\beta)\md \beta  \right]^2=0.
\end{align*}
Based on the condition of Cauchy-Schwartz inequality being equality, there exists $\lambda>0$ such that
\[
\alpha_+(\theta^*,\beta)\frac{x^{\alpha_+(\theta^*,\beta)-1}}{r^{\alpha_+(\theta^*,\beta)}}=\lambda \alpha_+(\theta^*,\beta)(1-\alpha_+(\theta^*,\beta))^2\frac{x^{\alpha_+(\theta^*,\beta)-1}}{r^{\alpha_+(\theta^*,\beta)}},
\]
for any $\beta \in {\rm supp}(p)$, $x\in (0,r)$. Because the map $\beta\mapsto \alpha_+(\theta^*,\beta)$ is injective, and $|\alpha_+(\theta^*,\beta)-1|$ remains the same for all $\beta\in {\rm supp}(p)$, we conclude that ${\rm supp}(p)$ contains at most two elements. If ${\rm supp}(p)=\{\beta_1,\beta_2\}$, we assume $p(\beta)=\lambda \delta_{\beta_1}+(1-\lambda)\delta_{\beta_2}$, $\lambda\in(0,1)$, then we have
\begin{align*}
\tilde{\theta}(x)&= \frac{\mu}{\sigma^2}\cdot \frac{\lambda \alpha_1 \frac{x^{\alpha_1-1}}{r^{\alpha_1}}+(1-\lambda)\alpha_2 \frac{x^{\alpha_2-1}}{r^{\alpha_2}} }{\lambda \alpha_1(1-\alpha_1) \frac{x^{\alpha_1-1}}{r^{\alpha_1}}+(1-\lambda)\alpha_2(1-\alpha_2) \frac{x^{\alpha_2-1}}{r^{\alpha_2}}} \\
&= \frac{\mu}{\sigma^2}\cdot \frac{\lambda \alpha_1 \frac{x^{\alpha_1-\alpha_2}}{r^{\alpha_1}}+(1-\lambda)\alpha_2 \frac{1}{r^{\alpha_2}} }{\lambda \alpha_1(1-\alpha_1) \frac{x^{\alpha_1-\alpha_2}}{r^{\alpha_1}}+(1-\lambda)\alpha_2(1-\alpha_2) \frac{1}{r^{\alpha_2}}}.
\end{align*}
Here we denote $\alpha_1=\alpha_+(\theta^*,\beta_1)$, $\alpha_2=\alpha_+(\theta^*,\beta_2)$ and without loss of generality assume $\alpha_1>\alpha_2$. It is clear that $\tilde{\theta}(0+)=\frac{\mu}{\sigma^2}\frac{1}{1-\alpha_2}$ and $\tilde{\theta}(r-)=\frac{\mu}{\sigma^2}\frac{\lambda \alpha_1+(1-\lambda)\alpha_2}{\lambda \alpha_1(1-\alpha_1)+(1-\lambda)\alpha_2(1-\alpha_2)}$. $\tilde{\theta}(r-)=\tilde{\theta}(0+)$ now gives $\alpha_1=\alpha_2$, a contradiction. Therefore, ${\rm supp}(p)$ is singleton. \qedhere
\end{proof}
{ 
By Proposition \ref{exmirrational} we conclude that there is no investment proportional to the wealth (constant equilibrium) that can serve as equilibrium strategy. The natural question is that is there any equilibrium strategy at all? In the rest of the present subsection, we give a positive answer in a special case where the discount function is two-point quasi exponential (see Remark \ref{discountexm}). Using the methodology proposed in this paper, we find that the equilibrium strategy is described by two coupled {\bf singular boundary value problems (sBVP)}. They are interesting and mathematically challenging in their own rights and the existence of positive solutions of such problems are dealt with in Appendix \ref{extsBVP}.

To see this, we pick
\[
\mB(t)=\frac{1}{2}e^{-\beta_1 t}+\frac{1}{2}e^{-\beta_2 t},
\]
with $0<\beta_1<\beta_2$\footnote{To avoid complicated notations, we choose two-point distribution and uniform weight (i.e., both discount rates appear with probability 1/2). It is straightforward to generalize the content of this part to the general finite discrete distributions:
\[
\mB(t)=\sum_{j=1}^N p_je^{-\beta_jt},
\]
where $\sum_j p_j=1$. But this generalization leads to complicated notations and tedious combinatorial discussions. Moreover, the equilibrium is described by an ($N-1$)-coupled system of singular boundary value problem, which is a generalization of (\ref{sBVP1}).}. Recall that when we choose the control strategy $\theta$, the dynamic of the wealth will be (see (\ref{wealthdynamic})):
\[
X^{\theta}_t=\mu\theta(X^{\theta}_t)X_t\md t+\sigma \theta(X^{\theta}_t)X_t \md W_t.
\]

Now for a candidate strategy $\theta$ (not necessarily constant) and a continuation region $(0,b)$, the corresponded value function is
\begin{equation}
\label{valuefunctionqexp}
f(s,t,x)=\frac{\log b}{2}\left(\mE^{x}e^{-\beta_1(\tau^{\theta}_b+t-s)}    +\mE^{x}e^{-\beta_2(\tau^{\theta}_b+t-s)} \right),0<x<b.
\end{equation}
where
\[
\tau^{\theta}_b=\inf\{t\geq 0: X^{\theta}_t=b\}.
\]
From one-dimensional diffusion theory we know that when the drift and diffusion function $x\mapsto x\theta(x)$ is Lipschitz (which we require), $0$ is an inaccessible boundary point (see \cite{ito1996diffusion} and \cite{Helland1996} for detailed illustrations). Therefore, under the stopping policy $C=(0,b)$ and control strategy $\theta$, the implemented stopping time is $\tau^{\theta}_b$, hence justifying the expression (\ref{valuefunctionqexp}). Using Corollary \ref{habitcoro}, we conclude that the pair $(\theta,b)$ is equilibrium if and only if $x\mapsto x\theta(x)$ is Lipschitz, and
\begin{align}
	&\theta(x)=-\frac{\mu (\phi'_1(x)+\phi'_2(x))}{x\sigma^2(\phi''_1(x)+\phi''_2(x))},0<x<b \label{qexpcon1}\\
	&\frac{\log b}{2}(\phi'_1(b)+\phi'_2(b))=\frac{1}{b},\label{qexpcon2}\\
	&\frac{\log b}{2}(\phi_1(x)+\phi_2(x))\geq \log x,0<x<b \label{qexpcon3}\\
	&\log b \geq \frac{\kappa}{\beta_1+\beta_2}\label{qexpcon4},
\end{align}
where $\phi_i(x)=\mE^x e^{-\beta_i \tau^\theta_b}$, $i=1,2$ and $\kappa=\frac{\mu^2}{\sigma^2}$. Meanwhile, from diffusion theory, we know that the function $\phi_i$ is the unique positive increasing solution of the following boundary value problem:
\begin{equation}
	\label{BVP1}
	\left\{
	\begin{aligned}
		&\frac{\sigma^2}{2}x^2\theta(x)^2\phi''_i+\mu x\theta(x)\phi'_i-\beta_i\phi_i=0,\\
		&\phi_i(0)=0,\phi_i(b)=1.
	\end{aligned}
	\right.
\end{equation}
Plugging (\ref{qexpcon1}) into (\ref{BVP1}) and rearranging the resulted problem equivalently, we have
\begin{equation}
	\label{BVP2}
	\left\{
	\begin{aligned}
		&\phi''_i+\kappa \frac{\phi'_1+\phi'_2}{\beta_1\phi_1+\beta_2\phi_2}\phi'_i-\frac{\kappa}{2}\beta_i\left(\frac{\phi'_1+\phi'_2}{\beta_1\phi_1+\beta_2\phi_2}\right)^2\phi_i=0,\\
		&\phi_i(0)=0,\phi_i(b)=1.
	\end{aligned}
	\right.
\end{equation}
The following simple lemma will be the first of some auxiliary results in this subsection.
\begin{lemma}
	\label{qelemma1}
	Suppose (\ref{BVP2}) has a positive solution $\phi_i$, $i=1,2$ which are strictly increasing on $(0,b)$. Then
	\[
	\tilde{\theta}(x)\triangleq \frac{2}{\mu}\frac{\beta_1\phi_1+\beta_2\phi_2}{\phi'_1+\phi'_2}
	\]
	is Lipschitz and $\theta(x)\triangleq \tilde{\theta}(x)/x$ satisfies (\ref{qexpcon1}) and (\ref{BVP1}).
\end{lemma}
\begin{proof}
	Summing the first equation in (\ref{BVP2}) over $i=1,2$, we get
	\[
	\phi''_1+\phi''_2+\frac{\kappa}{2}\frac{(\phi'_1+\phi'_2)^2}{\beta_1\phi_1+\beta_2\phi_2}=0.
	\]
	Therefore
	\[
	\tilde{\theta}(x)=-\frac{\kappa}{\mu}\frac{\phi'_1+\phi'_2}{\phi''_1+\phi''_2},
	\]
	and $\theta$ satisfies (\ref{qexpcon1}) and (\ref{BVP1}). On the other hand
	\begin{align*}
		\tilde{\theta}'(x)&=\frac{2}{\mu}\frac{(\beta_1\phi'_1+\beta_2\phi'_2)(\phi'_1+\phi'_2)-(\beta_1\phi_1+\beta_2\phi_2)(\phi''_1+\phi''_2)}{(\phi'_1+\phi'_2)^2}\\
		&=\frac{2}{\mu}\frac{(\beta_1\phi'_1+\beta_2\phi'_2)(\phi'_1+\phi'_2)+\frac{\kappa}{2}(\phi'_1+\phi'_2)^2}{(\phi'_1+\phi'_2)^2}\\
		&\in \left[0, \frac{2}{\mu}(\beta_2+\frac{\kappa}{2})\right].
	\end{align*}
Therefore $\tilde{\theta}$ is Lipschitz.
\end{proof}

In the following we focus on (\ref{BVP2}), which is a fully coupled nonlinear singular boundary problem. To simplify it we first use the following linear transformation:
\[
\psi_1(x)=\phi_1(x)+\phi_2(x),\psi_2(x)=\beta_1\phi_1(x)+\beta_2\phi_2(x).
\]
Under this transformation, (\ref{BVP2}) is transformed to
\begin{equation}
	\label{BVP3}
	\left\{
	\begin{aligned}
		&\psi''_1+\frac{\kappa}{2}\frac{(\psi'_1)^2}{\psi_2}=0,\\
		&\psi''_2+\kappa\frac{\psi'_1}{\psi_2}\psi'_2=\frac{\kappa}{2}\left((\beta_1+\beta_2)\frac{\psi'_1}{\psi_2}\psi'_1-\beta_1\beta_2\left(\frac{\psi'_1}{\psi_2}  \right)^2\psi_1    \right),\\
		&\psi_1(0)=\psi_2(0)=0,\\
		&\psi_1(b)=2,\psi_2(b)=\beta_1+\beta_2.	
	\end{aligned}
\right.
\end{equation}
We then consider the decoupling function $h$, i.e., we suppose { $\psi_2(x)=h(\psi_1(x))$}. By direct calculation, we can decouple (\ref{BVP3}) into two sBVP ($h=h(t)$):
\begin{equation}
	\label{sBVP1}
	\left\{
	\begin{aligned}
		&h''+\frac{\kappa}{2}\left[\frac{\beta_1\beta_2 t}{h^2}-\frac{\beta_1+\beta_2}{h}+\frac{h'}{h}    \right]=0,\\
		&h(0)=0,h(2)=\beta_1+\beta_2.
	\end{aligned}
	\right.
\end{equation}

\begin{equation}
	\label{sBVP2}
	\left\{
	\begin{aligned}
		&\psi''+\frac{\kappa}{2}\frac{(\psi')^2}{h(\psi)}=0,\\
		&\psi(0)=0,\psi(b)=2.
	\end{aligned}
	\right.
\end{equation}
\begin{remark}
	The trivialization of the case $\beta_1=\beta_2$ is through the sBVP (\ref{sBVP1}) because if $\beta=\beta_1=\beta_2$, it admits analytical solution $h(t)=\beta t$.
\end{remark}
It turns out that our existence result about equilibrium stopping and control strategies depends crucially on the existence of positive solutions together with some upper and lower bound estimation. We state this result in Lemma \ref{sBVPthm} below, whose proof relies on Leray-Schauder topological degree theory, and is postponed to Appendix \ref{extsBVP}.
\begin{lemma}
	\label{sBVPthm}
	(\ref{sBVP1}) has a solution $h\in C^1[0,2]\cap C^2(0,2)$ such that $\beta_1\leq h'(t)\leq \beta_2$, $\forall t\in [0,2]$. For any $b>0$, (\ref{sBVP2}) has a strictly increasing solution $\psi^b\in C[0,b]\cap C^2(0,b)$ such that $(x/b)^{\alpha_2}\leq \psi^b(x)/2\leq (x/b)^{\alpha_1}$, with $\alpha_i=2\beta_i/(2\beta_i+\kappa)$, $i=1,2$.
\end{lemma}
\begin{proof}
	See Appendix \ref{extsBVP}.
\end{proof}
(\ref{qexpcon1}) can now be verified using the results in Lemmas \ref{qelemma1} and \ref{sBVPthm}. The next lemma, on the other hand, deals with conditions (\ref{qexpcon2})-(\ref{qexpcon4}).

\begin{proposition}
	\label{qelemma3}
	Assume that
	\begin{equation}
		\label{qeqcond}
		\kappa<\min\{\frac{2\beta_1\beta_2}{\beta_2-\beta_1},\beta_1+\beta_2\}.
	\end{equation}
	Suppose that $h$ and $\psi^b$ are as in Lemma \ref{sBVPthm}. Then there exists $b>1$ such that (\ref{qexpcon2})-(\ref{qexpcon4}) hold. As a consequence, the strategy given by $(\theta^*,(0,b)\times (0,\infty))$ is an equilibrium, where
	\[
	\theta^*(x)=\left\{
	\begin{aligned}
	&\frac{2}{\mu x}\frac{h(\psi^b(x))}{(\psi^b)'(x)},&0<x<b,\\
	&\frac{(\beta_1+\beta_2)\log b}{\mu},&x\geq b.
	\end{aligned}
	\right.
	\]
\end{proposition}
\begin{proof}
	Using the equation and initial condition at $x=0$ in (\ref{sBVP2}), we have the following expression:
	\[ 
	\psi^b(x)=(\psi^b)'(b)\int_0^x \exp\left(\frac{\kappa}{2}\int_{\psi^b(y)}^2\frac{1}{h(z)}\md z\right) \md y.
	\]
	Because $\psi^b(b)=2$, we have
	\[ 
	(\psi^b)'(b)=2\left/\int_0^b \exp\left(\frac{\kappa}{2}\int_{\psi^b(y)}^2\frac{1}{h(z)}\md z\right) \md y.    \right.
	\]
	Therefore, $(\psi^b)'(b)=2/(b\log b)$ is equivalent to
	\[ 
	b\log b= G(b)\triangleq \int_0^b \exp\left(\frac{\kappa}{2}\int_{\psi^b(y)}^2\frac{1}{h(z)}\md z\right) \md y.
	\]
	By the fact that $h(t)\geq \beta_1 t$ and $\psi^b(x)/2\geq (x/b)^{\alpha_2}$ we have
	\begin{align*}
	G(b)& \leq \int_0^b \exp\left(\frac{\kappa}{2}\int_{\psi^b(y)}^2\frac{1}{\beta_1 z}\md z\right) \md y\\
	&\leq \int_0^b \left( \frac{b}{y}  \md y \right)^{\frac{\kappa \alpha_2}{2\beta_1}}\\
	&=\frac{b}{1-\frac{\kappa \alpha_2}{2\beta_1}}.
	\end{align*}
Here we use $ \frac{\kappa \alpha_2}{2\beta_1}<1$, which directly comes from (\ref{qeqcond}). That is to say, $G(b)$ is at most linear growth on $b$. Therefore for $b$ large enough, $b\log b>G(b)$. But for $b\leq 1$, $b\log b < 0<G(b)$. From continuity with respect to $b$ (see the proof of Proposition \ref{appDpro} in Appendix \ref{extsBVP}), we conclude that there exist a $b>1$ such that $b\log b=G(b)$, i.e., (\ref{qexpcon2}) holds. Using the same way, but estimating the lower bound, we get
\[
b\log b \geq \frac{b}{1-\frac{\kappa \alpha_1}{2\beta_2}}.
\]
Therefore
\[
\log b>1>\frac{\kappa}{\beta_1+\beta_2},
\]
due to (\ref{qeqcond}), thus (\ref{qexpcon4}) is satisfied. (\ref{qexpcon3}) is equivalent to
\[
\psi^b(x)\geq \frac{2\log x}{\log b}.
\]
We investigate $H(x)\triangleq \psi^b(x)-\frac{2\log x}{\log b}$. Note that
\begin{align*}
H'(x) & = (\psi^b)'(x)-\frac{2}{x\log b} \\
      & = \frac{2}{b \log b}\left( e^{\frac{\kappa}{2}\int_{\psi^b(x)}^2\frac{1}{h(z)}\md z} -\frac{b}{x}\right) \\
      &\leq  \frac{2}{b \log b}\left[\left(\frac{2}{\psi^b(x)}\right)^{\frac{\kappa}{2\beta_1}}-\frac{b}{x}  \right]  \\
      &\leq  \frac{2}{b \log b}\left[\left(\frac{b}{x}\right)^{\frac{\kappa \alpha_2}{2\beta_1}}-\frac{b}{x} \right] \\
      &<0.
\end{align*}
Therefore we conclude $H(x)\geq H(b)=0$ for any $0<x<b$, i.e., (\ref{qexpcon3}) holds. To show the rest of the proposition, we only need to prove $\phi_i$ are both ({ strictly}) increasing, because the diffusion theory then implies $\phi_i(x)$ is indeed identified with $\mE^x e^{\beta_i \tau_b}$. To do this, we notice that by the transformation $\psi^b(x)=\phi_1(x)+\phi_2(x)$, $h(\psi^b(x))=\beta_1\phi_1(x)+\beta_2\phi_2(x)$, we have
\begin{align*}
	&\phi_1(x)=\frac{\beta_2\psi^b(x)-h(\psi^b(x))}{\beta_2-\beta_1},\\
	&\phi_2(x)=\frac{h(\psi^b(x))-\beta_1\psi^b(x)}{\beta_2-\beta_1}.
\end{align*}
Therefore $\phi_1'(x)=(\beta_2-h'(\psi^b(x))(\psi^b)'(x))/(\beta_2-\beta_1)$, $\phi_2'(x)=(h'(\psi^b(x))-\beta_1)(\psi^b)'(x))/(\beta_2-\beta_1)$, which are both non-negative. { 
Suppose for some $x_0\in (0,b)$, $\phi_1'(x_0)=0$. Because $\psi^b$ is strictly increasing, $(\psi^b)'(x_0)=\phi_1'(x_0)+\phi_2'(x_0)=\phi_2'(x_0)>0$. By (\ref{BVP2}),
\[
\phi_1''(x_0)=\frac{\kappa\beta_1}{2}\left[\frac{(\phi_2'(x_0))^2}{\beta_1\phi_1(x_0)+\beta_2\phi_2(x_0)}   \right]\phi_1(x_0)>0.
\] 
This implies that for some $x'<x_0$, $\phi_1'(x')<0$, which is a contradiction. This implies that both $\phi'_i(x)>0$ for $x\in (0,b)$, $i=1,2$, completing the proof.	
}	
\end{proof}
\begin{remark}
	(\ref{qeqcond}) is very mild and is satisfied by typical model parameters. Moreover, it is only a technical assumption used to guarantee (\ref{qexpcon2})-(\ref{qexpcon4}) and is irrelevant to (\ref{qexpcon1}), which is usually the most important relation determining equilibrium control $\theta$. As illustrated at the beginning of this subsection, taking control into consideration brings essential differences and difficulties when studying time-inconsistent problems.
\end{remark}
\begin{remark}
	To the best of our knowledge, this is the first existence result of time-inconsistency problems where no explicit form is available and the diffusion coefficient is controlled. Existence of equilibrium has been widely acknowledged as challenging and open problems. Even when stopping and control are considered separately, the existence results (of general form equilibrium) are based on either specific model (LQ or diffusion coefficient uncontrolled) or restrictive technical assumptions (say, Lipschitz condition uniform in control). In this subsection, we give an existence result under a rather practical model and mild assumptions. Indeed, it is a little bit unfortunate that all of our arguments in this subsection only apply to finitely supported distributions. It is conjectured that we have a correspondence between discount function $\mB$ and ``mysterious function" $h$ (which is the solution to (\ref{sBVP1}) in the current situation). We choose to leave this as a direction for future work.
\end{remark}
}

{ 
\subsection{A two-dimensional example}
\label{2dexm}
In this subsection we give an example in two dimensions to illustrate possible applications of our theoretical results to a multi-dimensional setting. For simplicity, we assume that a controller controls the diffusion coefficients $\theta>0$ of two independent Brownian motions:
\[
\left\{
\begin{aligned}
	&X^{1,\theta}_t=\theta W^{1}_t,\\
	&X^{2,\theta}_t=\theta W^{2}_t.
\end{aligned}
\right.
\]
The reward function is the expected squared euclidean distance to origin, discounted by a hyperbolic function of rate $\beta$. To be specific, we choose
\[
g(s,t,x)=\frac{|x|^2}{1+\beta(t-s)},
\]
for $x\in \mR^2$. Using notations in Section \ref{notation}, we assume $\Ui=(0,\theta_0]$. In other words, the controller can only choose the diffusion coefficient from a bounded interval. Besides, he can choose a region $C\subset \mR^2$ and terminate the system once the state pair $(X^{1,\theta},X^{2,\theta})$ exits $C$. He will then receives
\[
J(t,x;\theta,\tau_C)=\mE^{t,x}\left[\frac{|X^{1,\theta}_{\tau_C}|^2+|X^{2,\theta}_{\tau_C}|^2}{1+\beta\tau_C}I_{\{\tau_C<\infty\}}\right].
\]
We denote by $R^0=\sqrt{|W^{(1)}|^2+|W^{(2)}|^2}$ the Bessel process of order 0, and $\tau_b=\inf\{s\geq 0: |X^{1,\theta}_s|^2+|X^{2,\theta}_s|^2\geq b^2    \}=\sigma_{b/\theta}\triangleq\inf\{s\geq 0: R^0_s\geq b/\theta \}$. As we have $\mP^{|x|/\theta}(R^0_r>0,\forall 0\leq r<\infty)=1$, we may choose $\mX=\mR^2\backslash\{(0,0)\}$, i.e., exclude the origin from state space. { For this example, we have the following equilibrium result:}
\begin{proposition} 
	Denote by $B_*(0,r)=\{x\in\mR^2:0<|x|<r \}$ the ball without its center. Then $(\theta_0,B_*(0,\theta_0\sqrt{t^*/2\beta})\times [0,\infty) )$ is an equilibrium pair, where $t^*$ is later determined in (\ref{dterminetstar}).
\end{proposition}
To prove this proposition, considering $\theta\in (0,\theta_0]$ and $C=B_*(0,b)$, we derive by direct computation  for $x\in C$,
\begin{align*}
f(s,t,x)&=\mE^{x}\left[ \frac{|X^{1,\theta}_{\tau^b}|^2+|X^{2,\theta}_{\tau_b}|^2}{1+\beta(\tau_b+t-s)} I_{\{\tau_b<\infty\}} \right]\\
&=\mE^{|x|/\theta}\left[\frac{\theta^2 |R^0_{\sigma_{b/\theta}}|^2}{1+\beta(\sigma_{b/\theta}+t-s)}    \right].
\end{align*}
In the calculation above, the indicator function $I_{\{\sigma_{b/\theta}<\infty\}}$ is neglected because $\mP^{|x|/\theta}(\sigma_{b/\theta}<\infty)=1$. Invoking formula (2.0.1) on page 297 of \cite{2002Handbook} and (\ref{hyperbolic}) with $a=b=\beta$, we further have
\[
f(s,t,x)=\frac{b^2}{\beta}\int_0^{\infty}e^{-r(t-s)-r/\beta}\frac{I_0\left(\frac{|x|\sqrt{2r}}{\theta}\right)}{I_0\left(\frac{b\sqrt{2r}}{\theta}\right)}\md r,0<|x|<b.
\]
Here and afterwards, we denote by $I_n$ the modified Bessel function of the first kind, with order $n\in \mN$. We will find a pair $(\theta,b)$ such that (\ref{sysinfg1})-(\ref{sysinfg6}) are true. By definition of $A^\theta$, we have $A^\theta h=\partial_t h+\frac{1}{2}\theta^2(\partial_{x_1x_1}h+\partial_{x_2x_2}h)$ when operating on functions $h$. Therefore for any $\theta'\in (0,\theta_0]$
\[
A^{\theta'}g(t,t,x)=-\beta |x|^2+2\theta'^2.
\]
On the other hand, it is straightforward to show that, for $i=1,2$ and $0<|x|<b$,
\begin{align*}
	\partial_t f(t,t,x)=&-\frac{b^2}{\beta}\int_0^{\infty}e^{-r/\beta}\frac{I_0\left(\frac{|x|\sqrt{2r}}{\theta}\right)}{I_0\left(\frac{b\sqrt{2r}}{\theta}\right)}r\md r,\\
	\partial_{x_i}f(t,t,x)=&\frac{b^2}{\beta}\int_0^{\infty}e^{-r/\beta}\frac{I'_0\left(\frac{|x|\sqrt{2r}}{\theta}\right)}{I_0\left(\frac{b\sqrt{2r}}{\theta}\right)}\frac{\sqrt{2r}}{\theta}\frac{x_i}{|x|}\md r,\\
	\partial_{x_ix_i}f(t,t,x)=&\frac{b^2}{\beta}\int_0^{\infty}e^{-r/\beta}\left[\frac{I''_0\left(\frac{|x|\sqrt{2r}}{\theta}\right)}{I_0\left(\frac{b\sqrt{2r}}{\theta}\right)}\frac{2r}{\theta^2}\frac{x^2_i}{|x|^2}+\frac{I'_0\left(\frac{|x|\sqrt{2r}}{\theta}\right)}{I_0\left(\frac{b\sqrt{2r}}{\theta}\right)}\frac{\sqrt{2r}}{\theta}\left(\frac{1}{|x|}-\frac{x_i^2}{|x|^3}\right)\right]\md r.
\end{align*}
Thus, we have
\begin{align*}
A^{\theta'}f(t,t,x)&=\partial_tf(t,t,x)+\frac{\theta'^2}{2}(\partial_{x_1x_1}f(t,t,x)+\partial_{x_2x_2}f(t,t,x))\\
&=\frac{b^2}{\beta}\int_0^{\infty}e^{-r/\beta}\left[ -\frac{I_0\left(\frac{|x|\sqrt{2r}}{\theta}\right)}{I_0\left(\frac{b\sqrt{2r}}{\theta}\right)}r+\frac{I''_0\left(\frac{|x|\sqrt{2r}}{\theta}\right)}{I_0\left(\frac{b\sqrt{2r}}{\theta}\right)}\left(\frac{\theta'}{\theta}\right)^2 r +\frac{I'_0\left(\frac{|x|\sqrt{2r}}{\theta}\right)}{I_0\left(\frac{b\sqrt{2r}}{\theta}\right)}\frac{\theta'^2}{\theta}\frac{\sqrt{r}}{\sqrt{2}|x|}\right] \md r	\\
&=\frac{b^2}{\beta}\int_0^{\infty}e^{-r/\beta}\frac{r}{I_0\left(\frac{b\sqrt{2r}}{\theta}\right)}\left( \frac{\theta'}{\theta}\right)^2\left(  I''_0\left(\frac{|x|\sqrt{2r}}{\theta}\right)+\frac{I'_0\left(\frac{|x|\sqrt{2r}}{\theta}\right)}{\frac{|x|\sqrt{2r}}{\theta}}-\left(\frac{\theta}{\theta'} \right)^2I_0\left(\frac{|x|\sqrt{2r}}{\theta}\right) \right) \md r 
\end{align*}
\begin{align*}
&=\frac{b^2}{\beta}\int_0^{\infty}e^{-r/\beta}\frac{rI_0\left(\frac{|x|\sqrt{2r}}{\theta}\right)}{I_0\left(\frac{b\sqrt{2r}}{\theta}\right)}\left( \frac{\theta'}{\theta}\right)^2\left( 1-\left( \frac{\theta}{\theta'}  \right)^2\right) \md r\\
&=\frac{b^2}{\beta}\left(\left( \frac{\theta'}{\theta}\right)^2-1   \right)\int_0^{\infty}e^{-r/\beta}\frac{rI_0\left(\frac{|x|\sqrt{2r}}{\theta}\right)}{I_0\left(\frac{b\sqrt{2r}}{\theta}\right)} \md r.
\end{align*}
Here we have use the fact that $I''_0(z)+I_0'(z)/z=I_0(z)$ for $z>0$. Obviously, choosing $\theta=\theta_0$ will make $f$ satisfy (\ref{sysinf2}), and (\ref{sysinfgplus}) will be true if
\begin{equation}\label{Agexm3}
-\beta b^2+2\theta_0^2\leq 0,
\end{equation}
which will be handled at last. Next, we claim that if $\partial_{x_i}f(t,t,b)=\partial_{x_i}g(t,t,b)$, i.e., the smooth fitting condition is verified, then $f\geq g$ inside $C$ so that (\ref{sysinfg6}) is also true. Define
\[
F(z)=\frac{b^2}{\beta}\int_0^{\infty}e^{-r/\beta}\frac{I_0\left(\frac{z\sqrt{2r}}{\theta}   \right)}{I_0\left(\frac{b\sqrt{2r}}{\theta}   \right)} \md r,z\in \mR.
\]
It now suffices to show $F(z)\geq z^2$ for $0\leq z\leq b$ because $f(t,t,x)\geq g(t,t,x)$ is equivalent to $F(|x|)\geq |x|^2$. By properties of modified Bessel functions of the first kind, we know any order derivatives of $F$ are positive. In particular, $G(z)=F(z)-z^2$ satisfies $G'''(z)\geq 0$. Thus $G''$ is increasing. As $\partial_{x_i} f(t,t,x)=\partial_{x_i}F(|x|)=F'(|x|)\frac{x_i}{|x|}$, $\partial_{x_i}g(t,t,x)=2x_i$, smooth fitting condition implies $F'(b)/b=2$, i.e., $G'(b)=F'(b)-2b=0$. Now combining the facts that $G''$ is increasing, $G''(0)=-2<0$ and $G'(0)=G'(b)=0$ we know $G'(z)\leq 0$ for $0\leq z\leq b$. Clearly $G(0)=F(0)>0$, $G(b)=0$, we thus conclude that $G(z)\geq 0$ for $0\leq z\leq b$, which in turn implies $f\geq g$ inside $C$. We now only need to determine $b$ such that smooth fitting condition and (\ref{Agexm3}) are true. By direct computation, smooth fitting condition translates to
\[
\frac{b}{\beta\theta}\int_0^{\infty}e^{-r/\beta}\frac{I'_0\left(\frac{b\sqrt{2r}}{\theta}   \right)}{I_0\left(\frac{b\sqrt{2r}}{\theta}   \right)}\sqrt{2r}\md r = 2,
\]
or equivalently (by a change of variable formula) $b/\theta=\sqrt{t^*/2\beta}$, where $t^*$ satisfies
\begin{equation}\label{dterminetstar}
\int_0^{\infty}e^{-u^2/t^*}\frac{I'_0(u)}{I_0(u)}\frac{u^2}{t^*} \md u =1.
\end{equation}
Numerical results show that $t^*\approx8.3419$. Noticing that $t^*>4$, we have $b^2/\theta^2= t^*/2\beta >2/\beta$, which yields
\[
-\beta|x|^2+2\theta^2\leq -\beta b^2+2\theta^2<0,
\]
which leads to (\ref{Agexm3}). To conclude, we have proved that in the present example, the equilibrium pair is $(\theta_0,B_*(0,\theta_0\sqrt{t^*/2\beta})\times [0,\infty) )$.
}

\section{Conclusion}
\label{conlude}
This paper provides a unified framework for the studying of time-inconsistent stopping-control problems, which has not been considered before. We define the equilibrium strategies and obtain an equivalent characterization based on an extended HJB system, providing a methodology to verify or exclude equilibrium. As applications, we propose an investment-withdrawal decision model, where the time-inconsistent decision makers are provided with both the opportunity to choose portfolios and the right to stop discretionarily. Two concrete examples are studied using the equilibrium theory established in this paper, and we can show the existence of equilibrium strategies respectively in these two examples. {  Finally, a two-dimensional example is also provided to illustrate applications of our theoretical framework in multi-dimensional case.}

There are also many other interesting yet unexplored topics for future research. An ongoing work by the authors will consider the existence of equilibrium solutions for stopping control problems under fairly general assumptions. { Generalizations of existence results in Subsection \ref{nonexp} to other non-exponential discount functions are also listed here as important open problems}.

The established framework can also be coupled with other topics in financial mathematics, such as a more complicated market model.

\vskip 20pt
{\bf Acknowledgements.}
The authors acknowledge the support from the National Natural Science Foundation of China (Grant No.11871036, No.12271290). The authors also thank the members of the group of Actuarial Sciences and Mathematical Finance at the Department of Mathematical Sciences, Tsinghua University for their feedbacks and useful conversations. The authors gratefully appreciate
Ravi P. Agarwal from Texas A\&M University-Kingsville, Guohui Guan from Renmin University of China and Kristoffer Lindensjö from Stockholm University for their useful discussions and suggestions. We are also particularly grateful to the two anonymous reviewers and the associated editor whose suggestions helped us to greatly improve the quality of the article.

\vskip 20pt
{\bf Data availability statement.}
Data sharing not applicable to this article as no datasets were generated or analyzed during the current study.

\vskip 10pt
\bibliographystyle{plainnat}
\bibliography{ref}

\appendix
\renewcommand{\theequation}{\thesection.\arabic{equation}}

\section{Proofs of results in Section \ref{eqchar}}
\label{proofsec3}
\begin{proof}[Proof of Lemma \ref{lm1}]
 We first introduce some notations that will be used. For Markov times $\tau_1$ and $\tau_2$ taking value in $[0,T]$, we define $\tau_1\oplus \tau_2 \triangleq \tau_1 + \tau_2\circ \theta_{\tau_1}$. Strong Markovian property of the Markov processes $(t,Y)$ implies
 \[
 \mE^{t,x}F(t\oplus \tau_1 \oplus \tau_2,Y_{t\oplus \tau_1 \oplus \tau_2}) = \mE^{t,x}\mE^{t\oplus \tau_1,Y_{t\oplus \tau_1}}F(t\oplus \tau_1 \oplus \tau_2,Y_{t\oplus \tau_1 \oplus \tau_2}),
 \]
 for Borel-measurable $F$. The fact $\tau_{(\hat{u},C,t)}=t\oplus \tau_{(\hat{u},C,0)}$ and strong Markovian property yield
 \begin{align*}
 J(t,x;\hat{u},\tau_{(\hat{u},C,t+\epsilon)})&= \mE^{t,x}g(t,x,t\oplus\epsilon \oplus \tau_{(\hat{u},C,0)},X_{t\oplus \epsilon \oplus \tau_{(\hat{u},C,0)}}) \\
 &=\mE^{t,x}\mE^{t+\epsilon,X_{t+\epsilon}}g(t,x,t\oplus\epsilon \oplus \tau_{(\hat{u},C,0)},X_{t\oplus\epsilon \oplus \tau_{(\hat{u},C,0)}}) \\
 &=\mE^{t,x}\mE^{t+\epsilon,X_{t+\epsilon}}g(t,x, \tau_{(\hat{u},C,t+\epsilon)},X_{ \tau_{(\hat{u},C,t+\epsilon)}}) \\
 &=\mE^{t,x}f(t,x,t+\epsilon,X_{t+\epsilon}).
 \end{align*}
 From now on, write $f(\cdot,\cdot)=f(t,x,\cdot,\cdot)$, for fixed $(t,x)\in E$. Consider $(t,x)\in C$. For $\delta>0$ sufficiently small, let $B_{t,x}(\delta)=(t,t+\delta^2)\times \{ x':\|x'-x\|<\delta \}\subset C$. Choose a cut-off function $\chi\in C^{\infty}_c(B_{t,x}(\delta))$ such that $0\leq \chi \leq 1$, $\chi\equiv 1$ on $B_{t,x}(\delta/2)$ and denote $f^0=f\chi \in C^{2}_c(E)$. Clearly $A^{\hat{u}}f^0(t,x)=A^{\hat{u}}f(t,x)$. Based on the definition of $A^{\hat{u}}$, we have
 \begin{eqnarray}
 \mE^{t,x}[f(t+\epsilon,X_{t+\epsilon})-f(t,x)]&=&\mE^{t,x}[f^0(t+\epsilon,X_{t+\epsilon})-f^0(t,x)]
 +(I) \nonumber \\
 &=&\epsilon A^{\hat{u}}f^0(t,x)+o(\epsilon)+(I) \nonumber\\
 &=&\epsilon A^{\hat{u}}f(t,x)+o(\epsilon)+(I),\label{littleep1}
 \end{eqnarray}
 where
 \begin{align*}
 |(I)|&\leq \mE^{t,x}|f(t+\epsilon,X_{t+\epsilon})-f^0(t+\epsilon,X_{t+\epsilon})|\\
      &\leq C\mE^{t,x}|f(t+\epsilon,X_{t+\epsilon})| I_{\{\|X_{t+\epsilon}-x\|\geq \delta/2\}} \\
      &\leq C\mE^{t,x} (1+\|(X_{t+\epsilon})\|)^\gamma I_{\{\|X_{t+\epsilon}-x\|\geq \delta/2\}}  \\
      &\leq C [(\mE^{t,x}(1+\sup_{0\leq s \leq T}\|X_s)\|)^{\gamma p'}]^{1/{p'}} \mP^{t,x}(\|X_{t+\epsilon}-x\|\geq \delta/2)^{1/p} \\
      &\leq C (\frac{2}{\delta})^{2\beta/p}(\mE^{t,x}\| X_{t+\epsilon} -x  \|^{2\beta})^{1/p} \\
      &\leq C\epsilon ^{\beta/p}=o(\epsilon)
 \end{align*}
 for some $1<p<\beta$. Here we have used the fact that $f\in L^{\infty}_{\rm poly}$, $X_t=x$, H$ \ddot{\rm o}$lder's inequality, and standard estimation in stochastic differential equations: under Assumption \ref{assumption1}, for any $\beta>0$, $t\leq s_1\leq s_2\leq T$,
 \begin{eqnarray}
  &&\mE^{t,x}\|X_{s_1}-X_{s_2}   \|^{\beta}\leq C_T(1+x^{\beta}) |s_1-s_2|^{\beta/2}, \nonumber\\
 &&\mE^{t,x} \sup_{t\leq s\leq T}\|X_s\|^{\beta} \leq C_T(1+x^{\beta}).\label{littleep2}
  \end{eqnarray}
 See \citet*{Yong1999}. For $(t,x)\in \intt(D)$, we still use the cut-off technique near $(t,x)$ to make $f^0\in C^2_c(\intt(D))$ and all other things remain the same.
\end{proof}
\begin{proof}[Proof of Lemma \ref{lm2}]
Based on the definition of $J(t,x;u,\tau)$, we have
\begin{eqnarray}
&&J(t,x;\hat{u}_{(t,\epsilon,\bu)},\tau_{(\hat{u}_{(t,\epsilon,\bu)},C,t)})\nonumber\\
       &&=\mE^{t,x}g(t,x,\tau_{(\hat{u}_{(t,\epsilon,\bu)},C,t)},X^{\hat{u}_{(t,\epsilon,\bu)}}_{\tau_{(\hat{u}_{(t,\epsilon,\bu)},C,t)}})I_{\{\tau_{(\hat{u}_{(t,\epsilon,\bu)},C,t)}\leq t+\epsilon\}} \nonumber\\
       &&+\mE^{t,x}g(t,x,\tau_{(\hat{u}_{(t,\epsilon,\bu)},C,t)},X^{\hat{u}_{(t,\epsilon,\bu)}}_{\tau_{(\hat{u}_{(t,\epsilon,\bu)},C,t)}})I_{\{\tau_{(\hat{u}_{(t,\epsilon,\bu)},C,t)}> t+\epsilon\}} \nonumber\\
       &&=(I)+(II).\label{est0}
\end{eqnarray}
Noting that $\hat{u}_{(t,\epsilon,\bu)} \big|_{[t,t+\epsilon)} \equiv \bu$ (see (\ref{perturbation})), we conclude $\{\tau_{(\hat{u}_{(t,\epsilon,\bu)},C,t)}\leq t+\epsilon\}\subset \{\tau_{(\bu,C,t)}\leq t+\epsilon   \}$, and for some $\delta >0$ with $B_{t,x}(\delta)\subset C$,
\begin{align*}
|(I)|&\leq \mE^{t,x}g(t,x,\tau_{(\bu,C,t)},X^{\bu}_{\tau_{(\bu,C,t)}}) I_{\{\tau_{(\bu,C,t)}\leq t+\epsilon\}} \\
     &\leq C_{\delta}\mP^{t,x}(\sup_{t\leq s\leq t+\epsilon}\|X^{\bu}_s-x \|\geq \delta)^{1/p}.
\end{align*}
To estimate $\mP^{t,x}(\sup_{t\leq s\leq t+\epsilon}\|X^{\bu}_s-x\| \geq \delta )$, we conclude from Theorem 6.3, Chapter 1 in \citet*{Yong1999} that under the case of constant initial condition, we can take $\beta>4$. As such, using arguments in the proof of Lemma 3.1 of \citet*{Huang2019}, we have
\[
\mP^{t,x}(\sup_{t\leq s\leq t+\epsilon}\|X^{\bu}_s-x\|\geq \delta)\leq C_{\delta}\epsilon^{\gamma'}
\]
for $1<\gamma'<\beta/2-1$. Thus, picking $1<p<\gamma'$, we have
\begin{equation}
\label{est1}
|(I)|=o(\epsilon).
\end{equation}
On the other hand, it is clear from definition that $\{\tau_{(\hat{u}_{(t,\epsilon,\bu)},C,t)}> t+\epsilon\} \subset \{  \tau_{(\hat{u}_{(t,\epsilon,\bu)},C,t)}=\tau_{(\hat{u}_{(t,\epsilon,\bu)},C,t+\epsilon)}\} $. Moreover, for any $y\in \mX$,
\begin{equation}
\label{equid}
(\tau_{(\hat{u}_{(t,\epsilon,\bu)},C,t+\epsilon)},X^{\hat{u}_{(t,\epsilon,\bu)}}_{\tau_{(\hat{u}_{(t,\epsilon,\bu)},C,t+\epsilon)}} ) = (\tau_{(\hat{u},C,t+\epsilon)},X_{\tau_{(\hat{u},C,t+\epsilon)}}), \mP^{t+\epsilon,y}-{\rm a.s.}.
\end{equation}
Now, using (\ref{equid}) and Markovian property (of $(t,X^{\hat{u}_{(t,\epsilon,\bu)}})$), and conditioning on $\F_{t+\epsilon}$ if necessary, we have
\begin{align*}
(II)&= \mE^{t,x}(\mE^{t+\epsilon,X^{\bu}_{t+\epsilon}}g(t,x,\tau_{(\hat{u}_{(t,\epsilon,\bu)},C,t+\epsilon)},X^{\hat{u}_{(t,\epsilon,\bu)}}_{\tau_{(\hat{u}_{(t,\epsilon,\bu)},C,t+\epsilon)}})     )I_{\{\tau_{(\hat{u}_{(t,\epsilon,\bu)},C,t)}> t+\epsilon\}} \\
    &=\mE^{t,x}(\mE^{t+\epsilon,X^{\bu}_{t+\epsilon}}g(t,x,\tau_{(\hat{u},C,t+\epsilon)},X_{\tau_{(\hat{u},C,t+\epsilon)}})     )I_{\{\tau_{(\hat{u}_{(t,\epsilon,\bu)},C,t)}> t+\epsilon\}} \\
    &=\mE^{t,x}f(t,x,t+\epsilon,X^{\bu}_{t+\epsilon})I_{\{\tau_{(\hat{u}_{(t,\epsilon,\bu)},C,t)}> t+\epsilon\}}\\
    &=\mE^{t,x}f(t,x,t+\epsilon,X^{\bu}_{t+\epsilon})+o(\epsilon).
\end{align*}
The indicator function can be ignored by the same reason of estimation (\ref{est1}). Using the similar cut-off technique as in the proof of Lemma \ref{lm1}, we have
\begin{equation}
\label{est2}
\mE^{t,x}[f(t,x,t+\epsilon,X^{\bu}_{t+\epsilon})-f(t,x,t,x)]=\epsilon A^{\bu}f(t,x,t,x)+o(\epsilon).
\end{equation}
Combining (\ref{est0}), (\ref{est1}) and (\ref{est2}), Lemma \ref{lm2} follows.
\end{proof}

\begin{proof}[Proof of Lemma \ref{lm3}]
 (\ref{sys4}) and (\ref{sys5}) are clear from definition of $\tau_{(u,C,t)}$. To show (\ref{sys1}), using strong Markovian property, we have
\begin{align*}
\mE^{t,x}f(s,y,t+\epsilon,X_{t+\epsilon}) &=\mE^{t,x}\mE^{t+\epsilon,X_{t+\epsilon}}g(s,y,\tau_{(\hat{u},C,t+\epsilon)},X_{\tau_{(\hat{u},C,t+\epsilon)}})\\
&=\mE^{t,x}g(s,y,\tau_{(\hat{u},C,t+\epsilon)},X_{\tau_{(\hat{u},C,t+\epsilon)}}).
\end{align*}
Choosing $B_{t,x}(\delta)\subset C$ and using the same argument as in the proof of Lemma \ref{lm1}, we have
\begin{align*}
\epsilon A^{\hat{u}}f(s,y,t,x) &= \mE^{t,x}[f(s,y,t+\epsilon,X_{t+\epsilon})-f(s,y,t,x)]+o(\epsilon) \\
& =\mE^{t,x}[g(s,y,\tau_{(\hat{u},C,t+\epsilon)},X_{\tau_{(\hat{u},C,t+\epsilon)}}) -
g(s,y,\tau_{(\hat{u},C,t)},X_{\tau_{(\hat{u},C,t)}})  ] +o(\epsilon) \\
&=(I)+o(\epsilon).
\end{align*}
Noting that $\{ \tau_{(\hat{u},C,t)}>t+\epsilon  \}\subset \{\tau_{(\hat{u},C,t)}= \tau_{(\hat{u},C,t+\epsilon)}\}$, and using the similar techniques as in the proof of Lemma \ref{lm2}, we have $|(I)|=o(\epsilon)$, which leads to (\ref{sys1}).
\end{proof}
\vskip 5pt
To proceed with the proof of Theorem \ref{boundary}, we need the following technical lemmas. For simplicity, we write $B_{\delta}=\overline{B_{t,x}(\delta)}$ and define
\[
\triangle(s,y;\epsilon)\triangleq \frac{\mE^{s,y}[f(s+\epsilon,X_{s+\epsilon})-f(s,y)]}{\epsilon}.
\]
\begin{lemma}
\label{uniformconvergence}
For any compact subset $K\subset E\backslash \partial C$, $\displaystyle\lim_{\epsilon \to 0}\triangle(s,y;\epsilon)= A^{\hat{u}}f(s,y)$, uniformly in $(s,y)\in K$.
\end{lemma}
\begin{proof}
  We choose another compact $K'$ such that $K\subset K'\subset E\backslash \partial C$, and denote the cut-off of $f$ on $K'$ by $f^0$. For any $\eta>0$, $(s,y)\in K$, we have the following estimates:
  \begin{align*}
    |\triangle(s,y;\epsilon)-A^{\hat{u}}f(s,y)| & \leq \frac{1}{\epsilon}\mE^{s,y}\int_s^{s+\epsilon}|A^{\hat{u}}f^0(r,X_r)-A^{\hat{u}}f^0(s,y)|\md r +o(1)\\
    & \leq \sup_{s\leq r\leq s+\epsilon}\mE^{s,y}|A^{\hat{u}}f^0(r,X_r)-A^{\hat{u}}f^0(s,y)| + o(1) \\
    & \leq \sup_{\substack{s\leq r\leq s+\epsilon \\ |y-y'|<\eta \\ y,y'\in K'}}|A^{\hat{u}}f^0(r,y')-A^{\hat{u}}f^0(s,y)|+(I),
  \end{align*}
  where
  \begin{align*}
    (I) &=\sup_{s\leq r\leq s+\epsilon}\mE^{s,y}|A^{\hat{u}}f^0(r,X_r)-A^{\hat{u}}f^0(s,y)|I_{\{\|X_r-y\|\geq \eta\}}     \\
     & \leq C_{K,K'}\mP^{s,y}(\sup_{s\leq r\leq s+\epsilon}\|X_r-y\|\geq \eta)\\
     & \leq o(1).
  \end{align*}
  We emphasize that the $o(1)$ in the last inequality is uniform for $(s,y)\in K$. Combining the last two estimates, we have
  \[
    \lim_{\epsilon\to 0}\sup_{(s,y)\in K}|\triangle(s,y;\epsilon)-A^{\hat{u}}f(s,y)|  \leq \sup_{\substack{|y-y'|<\eta \\ y,y'\in K'}}|A^{\hat{u}}f^0(s,y')-A^{\hat{u}}f^0(s,y)| .
  \]
  Letting $\eta\to 0$, and using uniform continuity on $K'$, we complete the proof.
\end{proof}

\begin{lemma}
\label{boundlimsup}
  Assume that (\ref{fcond}), (\ref{fcon2}) and (\ref{spatialfitting}) hold. Then for any $(t,x)\in \partial C$, $\delta>0$,
  \begin{equation}
  \label{deltabound}
  \limsup_{\epsilon \to 0}\sup_{(s,y)\in B_{t,x}(\delta)}\triangle(s,y;\epsilon)\leq C(t,x,\delta) < \infty.
  \end{equation}
\end{lemma}
\begin{proof}
  First assume that (\ref{spatialfitting}) holds. We consider the Taylor expansion of $f(s',y')-f(s,y)$ when $d((s',y'),(s,y))$ is sufficiently small. When $(s,y)\in C$ or $(s,y)\in \intt(D)$, we can assume that both $(s,y)$ and $(s',y')$ are in $C$ or $\intt(D)$, and $f$ is $C^{1,2}$, as such, Taylor expansion is directly applicable. Let us focus on $(s,y)\in \partial C$. For $(s',y')\in C$, using Taylor expansion for $\tilde{f}$, we have
  \begin{equation}
  \label{TaylorC}
  f(s',y')-f(s,y)=\tilde{f}_x(s,y)(y'-y)+O(|s'-s|+\|y'-y\|^2).
  \end{equation}
  For $(s',y')\in D$, we use Taylor expansion for $g$ to get
  \begin{equation}
  \label{TaylorD}
  f(s',y')-f(s,y)=g_x(s,y)(y'-y)+O(|s'-s|+\|y'-y\|^2).
  \end{equation}
  Using locally boundedness of $f_t$, $f_{xx}$, $g_t$ and $g_{xx}$, we have $O(|s'-s|+\|y'-y\|^2)$ in (\ref{TaylorC}) and (\ref{TaylorD}) are uniform in $(s,y)\in B_{\delta}$. Now, using the usual cut-off techniques, we only consider $X_{t+\epsilon}\in B_{\delta}$. Combining (\ref{TaylorC}), (\ref{TaylorD}) and (\ref{spatialfitting}), we have
  \[
  f(s+\epsilon,X_{s+\epsilon})-f(s,y)=\tilde{f}_x(s,y)(X_{s+\epsilon}-y)+O(\epsilon+\|X_{s+\epsilon}-y\|^2).
  \]
  As
  \begin{align*}
    \mE^{s,y}(X_{s+\epsilon}-y) & = \mE^{s,y}\int_s^{s+\epsilon} \Theta(r,X_r)\md r \\
    &=O(\epsilon),
  \end{align*}
  using the above estimates and (\ref{littleep2}), we have the desired conclusion.
\end{proof}
\begin{lemma}
\label{inflm}
  Suppose that $Y$ with $Y_t=0$ is a continuous semimartingale such that
  \begin{equation}
  \label{nondege}
  \lim_{\epsilon\to 0}\mE\langle Y \rangle_{t+\epsilon}/\epsilon =\sigma>0,
  \end{equation}
  and
  \begin{equation}
  \label{momentest}
  \mE|Y_{t+\epsilon}|^4\leq C\epsilon^2,\ \ \forall \epsilon>0.
  \end{equation}
  Then
  \[
  \lim_{\epsilon\to 0}\mE|Y_{t+\epsilon}|/\epsilon = \infty.
  \]
\end{lemma}
\begin{proof}
  For any $\delta >0$, we have
  \begin{equation}
  \label{est8}
  \begin{aligned}
    \mE|Y_{t+\epsilon}|/\epsilon & \geq \mE|Y_{t+\epsilon}|I_{\{|Y_{t+\epsilon}|\leq \delta\}}/\epsilon \\
     & \geq \frac{1}{\delta\epsilon}\mE |Y_{t+\epsilon}|^2I_{\{|Y_{t+\epsilon}|\leq \delta\}}.
     \end{aligned}
  \end{equation}
  Because $Y$ is pathwise continuous, it is continuous in probability. As such, using (\ref{momentest}), we have
  \begin{align*}
    \mE|Y_{t+\epsilon}|^2I_{\{|Y_{t+\epsilon}|> \delta\}}/\epsilon & \leq \frac{1}{\epsilon}(\mE|Y_{t+\epsilon}|^4)^{1/2}\mP(|Y_{t+\epsilon}|\geq \delta)^{1/2} \\
    & \leq C\mP(|Y_{t+\epsilon}|\geq \delta)^{1/2} \\
    & =o(1),
  \end{align*}
  as $\epsilon\to 0$. Thus,
  \[
  \lim_{\epsilon\to 0}\mE |Y_{t+\epsilon}|^2I_{\{|Y_{t+\epsilon}|\leq \delta\}}/\epsilon = \lim_{\epsilon\to 0}\mE |Y_{t+\epsilon}|^2/\epsilon=\lim_{\epsilon\to 0}\mE \langle Y\rangle_{t+\epsilon}/\epsilon=\sigma>0.
  \]
  Letting $\epsilon\to 0$ on both sides of (\ref{est8}), we have
  \[
  \liminf_{\epsilon\to 0}\mE|Y_{t+\epsilon}|/\epsilon \geq \frac{\sigma}{\delta},\forall \delta>0.
  \]
  Letting $\delta \to 0$ gives the desired result.
\end{proof}

We are now ready to give the proof of sufficiency part of Theorem \ref{boundary}.
\begin{proof}[Proof of Theorem \ref{boundary}, (1)]
   Based on Lemma \ref{boundlimsup}, it is sufficient to prove that if (\ref{deltabound}) and (\ref{boundarysc}) are ture, then (\ref{sys3}) holds. Pick a sequence of $\epsilon_k\to 0$, such that $\lim_{k\to \infty} \triangle(t,x;\epsilon_k)=\limsup_{\epsilon \to 0}\triangle(t,x;\epsilon)$. For simplicity, we still denote this limit by $\epsilon \to 0$. Fix $\delta>0$ for now. For each $\epsilon$, using continuity of $\triangle(\cdot,\cdot;\epsilon)$, we find a $(s_{\epsilon},y_{\epsilon})\in B_{\delta}$ such that
  \[
  \triangle(s_{\epsilon},y_{\epsilon};\epsilon)-\frac{\delta}{d((s_{\epsilon},y_{\epsilon}),\partial C)+\epsilon} = \sup_{(s,y)\in B_{\delta}}\left[\triangle(s,y;\epsilon)-\frac{\delta}{d((s,y),\partial C)+\epsilon}     \right].
  \]
  Picking a subsequence, we assume $(s_{\epsilon},y_{\epsilon})\to (s_{\delta},y_{\delta})$. If $(s_{\delta},y_{\delta})\in \partial C$, fixing another $(s_{\delta}',y_{\delta}')\in B_{\delta}\backslash \partial C$ such that $d_{\delta}\triangleq d((s_{\delta}',y_{\delta}'),\partial C) >0 $, letting $\epsilon \to 0$ on both sides of
  \[
  \triangle(s_{\delta}',y_{\delta}';\epsilon)-\frac{\delta}{\epsilon+d_{\delta}}\leq \triangle(s_{\epsilon},y_{\epsilon};\epsilon)-\frac{\delta}{\epsilon+d((s_{\epsilon},y_{\epsilon}),\partial C)},
  \]
  and using the boundedness of $\limsup_{\epsilon \to 0}\triangle(s_{\epsilon},y_{\epsilon};\epsilon)$ from (\ref{deltabound}), we have \[A^{\hat{u}}f(s'_{\delta},y'_{\delta})-\frac{\delta}{d_{\delta}}\leq -\infty,
  \]
  yielding a contradiction. Thus, we can choose a $h_{\delta}>0$ sufficiently small, such that
  \[
  \sup_{(s,y)\in B_{\delta}}\left[\triangle(s,y;\epsilon)-\frac{\delta}{d((s,y),\partial C)+\epsilon}     \right]=\sup_{(s,y)\in B_{\delta}\backslash D_{h_{\delta}}}\left[\triangle(s,y;\epsilon)-\frac{\delta}{d((s,y),\partial C)+\epsilon}     \right],
  \]
  where
  \[
  D_{\eta}\triangleq \{ (s,y):d((s,y),\partial C)< \eta, \ \ \eta >0\}.
  \]
  Using Lemma \ref{uniformconvergence}, we conclude
  \begin{eqnarray*}
    &&\lim_{\epsilon \to 0}\sup_{(s,y)\in B_{\delta}}\left[\triangle(s,y;\epsilon)-\frac{\delta}{d((s,y),\partial C)+\epsilon}     \right]\\
    &=&\lim_{\epsilon \to 0}\sup_{(s,y)\in B_{\delta}\backslash D_{h_{\delta}}}\left[\triangle(s,y;\epsilon)-\frac{\delta}{d((s,y),\partial C)+\epsilon}     \right]  \\
    &=& \sup_{(s,y)\in B_{\delta}\backslash D_{h_{\delta}}}[A^{\hat{u}} f(s,y)-\frac{\delta}{d((s_{\delta},y_{\delta}),\partial C)}] \\
     & \leq&  \sup_{(s,y)\in B_{\delta}\backslash \partial C}A^{\hat{u}}f(s,y).
  \end{eqnarray*}
  As such, using (\ref{boundarysc}), we have
  \begin{eqnarray*}
  &&\lim_{\delta \to 0}\lim_{\epsilon \to 0}\sup_{(s,y)\in B_{\delta}}\left[\triangle(s,y;\epsilon)-\frac{\delta}{d((s,y),\partial C)+\epsilon}     \right]\\
  &\leq &  \lim_{\delta \to 0}\sup_{(s,y)\in B_{\delta}\backslash \partial C}A^{\hat{u}}f(s,y) \\
         &\leq & \limsup_{(s,y)\notin \partial C,(s,y)\to (t,x)}A^{\hat{u}}f(s,y) \\
         &\leq  &0.
  \end{eqnarray*}
  On the other hand, for any $\delta >0$,
  \begin{align*}
  \lim_{\epsilon \to 0}\triangle(t,x;\epsilon) &= \lim_{\epsilon \to 0}\lim_{\delta \to 0}\left[\triangle(t,x;\epsilon)-\frac{\delta}{\epsilon}\right]  \\
  &\leq \lim_{\epsilon \to 0} \lim_{\delta \to 0}\sup_{(s,y)\in B_{\delta}}\left[\triangle(s,y;\epsilon)-\frac{\delta}{d((s,y),\partial C)+\epsilon}     \right] \\
  &\leq \lim_{\epsilon \to 0}\sup_{(s,y)\in B_{\delta}}\left[\triangle(s,y;\epsilon)-\frac{\delta}{d((s,y),\partial C)+\epsilon}     \right].
  \end{align*}
  Letting $\delta \to 0$ on the right hand side of the last inequality gives
  \[
  \lim_{\epsilon \to 0}\triangle(t,x;\epsilon)\leq 0,
  \]
completing the proof.
\end{proof}
{ 
To provide a proof of necessity part of Theorem \ref{boundary}, we need to invoke a local time formula on surfaces (c.f., \citet*{Peskir2007}) because no $C^1$ regularity of $f$ across $\partial C$ is guaranteed and thus standard It$\hato$'s formula is not applicable. However, local time formula needs the boundary to be the graph of a function with certain properties. In our work, it is sufficient to express $\partial C$ locally as graph of a smooth function, based on the prescribed regularity and implicit function theorem. However this can only be done along directions not in tangent space. To resolve this we first propose a weaker version of smooth fitting, prove that it is a necessary condition of (\ref{sys3}) and finally show that it is equivalent to (\ref{spatialfitting}). To this end, we need to introduce some notations. For $(t,x)=(t,x^1,x^2,\cdots,x^n)\in E\subset \mR^{n+1}$, we write $(t,x)=te_t+\sum_{i=1}^n x^i e_i$, where $\{e_t,e_1,\cdots,e_n\}$ is canonical orthogonal basis. For $(t,x)\in \partial C$, denote the unit normal vector by $n(t,x)$. Moreover, for $(t,x)\in \partial C$, we define:
\begin{align*}
	&\mK(t,x)\triangleq \{k=1,2,\cdots,n: \langle e_k,n(t,x) \rangle \neq 0  \}, \\
	&\mathring{\partial C}\triangleq \{(t,x)\in \partial C:\mK(t,x)\neq \varnothing      \}.
\end{align*}
We propose another version of smooth fitting:
\begin{equation}\label{smoothfit}
	\frac{\partial f}{\partial e_k}(t,x,\cdot,\cdot)\big|_{(t,x)}=\frac{\partial g}{\partial e_k}(t,x,\cdot,\cdot)\big|_{(t,x)},\forall (t,x)\in \mathring{\partial C},k\in \mK(t,x).
\end{equation}
\begin{lemma}\label{SFequivalent}
	If (\ref{fcond}) and (\ref{fcon2}) hold, then (\ref{smoothfit}) is equivalent to (\ref{spatialfitting}).
\end{lemma}
\begin{proof}
	Because (\ref{spatialfitting}) $\Longrightarrow$ (\ref{smoothfit}) is by definition, we only prove (\ref{smoothfit}) $\Longrightarrow$ (\ref{spatialfitting}). To this end, assume $f_x(t,x,\cdot,\cdot)\big|_{(t,x)}=g_x(t,x,\cdot,\cdot)\big|_{(t,x)}$ for some $(t,x)\in \partial C$. We fix this $(t,x)$ from now on, and denote $F(t',x')=\tilde{f}(t,x,t',x')-g(t,x,t',x')$. From the assumption we know $\nabla F\neq 0$ at $(t,x)$. Using implicit function theorem, there exists a small ball $B$ such that $\{F=0\}\cap B$ is a $C^1$ hyper-surface and $n(t,x)= \nabla F(t,x)\}$. Thus we have for each $k\notin \mK(t,x)$, $\langle e_k,\nabla F(t,x)\rangle=0$, which is to say, $\frac{\partial \tilde{f}}{\partial e_k}(t,x,\cdot,\cdot)\big|_{(t,x)}= \frac{\partial g}{\partial e_k}(t,x,\cdot,\cdot)\big|_{(t,x)}$. Because $\nabla F(t,x)\neq 0$, there is at least one $k_0$ such that $\frac{\partial \tilde{f}}{\partial e_{k_0}}(t,x,\cdot,\cdot)\big|_{(t,x)}\neq  \frac{\partial g}{\partial e_{k_0}}(t,x,\cdot,\cdot)\big|_{(t,x)}$. Therefore we conclude $k_0\in \mK(t,x)$ and thus (\ref{smoothfit}) fails. This proves the implication (\ref{spatialfitting}) $\Longrightarrow$ (\ref{smoothfit}).
\end{proof}
Based on Lemma \ref{SFequivalent}, to show necessity in Theorem \ref{boundary}, we only need to show  (\ref{sys3})$\Longrightarrow$ (\ref{smoothfit})}.
\begin{proof}[Proof of Theorem \ref{boundary}, (2)]
We first claim that for any $(t,x)\in \partial C$, $k\in \mK(t,x)$, $\frac{\partial f}{\partial e_k}(t,x+)\geq\frac{\partial f}{\partial e_k}(t,x-)$. Indeed, based on regularity of $\partial C$ and definition of $\mK(t,x)$, there exists $h_0>0$ such that $(t,x)+he_k\in C$ and $(t,x)-he_k\in \intt(D)$, $\forall 0<h<h_0$, or $(t,x)+he_k\in \intt(D)$ and $(t,x)-he_k\in C$, $\forall 0<h<h_0$. If the former holds, for $h$ small enough, because $f=g$ in $\intt(D)$, we have
\begin{align*}
	\frac{\partial f}{\partial e_k}(t,x+)&=\lim_{h\to 0}(f((t,x)+he_k)-f(t,x))/h \\
	&\geq \lim_{h\to 0}(g((t,x)+he_k)-g(t,x))/h  \\
	&=\frac{\partial g}{\partial e_k}(t,x+) \\
	&=\frac{\partial g}{\partial e_k}(t,x-) \\
	&=\frac{\partial f}{\partial e_k}(t,x-).
\end{align*}
If the latter holds, for $h$ small enough,
\begin{align*}
	\frac{\partial f}{\partial e_k}(t,x+)&=\frac{\partial g}{\partial e_k}(t,x+) \\
	&=\frac{\partial g}{\partial e_k}(t,x-)  \\
	&=\lim_{h\to 0}(g((t,x)-he_k)-g(t,x))/(-h) \\
	&\geq \lim_{h\to 0}(f((t,x)-he_k)-f(t,x))/(-h) \\
	&=\frac{\partial f}{\partial e_k}(t,x-).
\end{align*}
Moreover, we conclude from the above argument that if $\frac{\partial f}{\partial e_k}(t,x+)=\frac{\partial f}{\partial e_k}(t,x-)$, then $\frac{\partial f}{\partial e_k}(t,x)=\frac{\partial g}{\partial e_k}(t,x)$. To prove (2) of Theorem \ref{boundary}, it is sufficient to obtain contradiction if we assume $\frac{\partial f}{\partial e_k}(t,x+)>\frac{\partial f}{\partial e_k}(t,x-)$ for some $(t,x)\in \mathring{\partial C}$, $k\in \mK(t,x)$. Because $\partial C$ is $C^2$, by definition, it can be locally expressed by $\{P(s,y)=0  \}$ for some $C^2$ function $P$. If $k\in \mK(s,y)$, it is clear that $ \frac{\partial }{\partial e_k}P(s,y)\neq 0$. Based on implicit function theorem, $\partial C$ can be expressed locally near $B_{\delta}$ by $y^k = p(t,y^1,\cdots,y^{k-1},y^{k+1},\cdots,y^n)$, where $p$ is $C^2$ function defined in a neighbourhood of $B_{\delta}$. Writing $X=(X^1,\cdots,X^n)$, we define $p^k=p(\cdot, X^1,\cdots,X^{k-1},X^{k+1},\cdots,X^n)$, which is a continuous semimartingale by regularity of $p$. Now using Peskir's local time formula (see \citet*{Peskir2007}), we have
\begin{align}
	\ \ \ \ \ \ \  f^0(s+\epsilon,X_{s+\epsilon})-f^0(s,y)&= \frac{1}{2}\sum_{i=1}^n\int_s^{s+\epsilon}\left\{ \frac{\partial f^0}{\partial e_i}(r,\cdots,X^k_r+,\cdots,X^n_r) \right.\nonumber\\
	&\left.+\frac{\partial f^0}{\partial e_i}(r,\cdots,X^k_r-,\cdots,X^n_r)\right\}\md X^i_r \nonumber \\
	& +\frac{1}{2}\int_s^{s+\epsilon}\left\{ \frac{\partial f^0}{\partial t}(r+,X_r) +\frac{\partial f^0}{\partial t}(r-,X_r)\right\}\md r \nonumber\\
	& +\frac{1}{4}\sum_{i,j}\int_s^{s+\epsilon}\left\{   \frac{\partial^2 f^0}{\partial e_i\partial e_j}(r,\cdots,X^k_r+,\cdots,X^n_r) \right.\nonumber
	\end{align}
    \begin{align}
	&\left.+ \frac{\partial^2 f^0}{\partial e_i\partial e_j}(r,\cdots,X^k_r-,\cdots,X^n_r)\right\}\md \langle X^i,X^j \rangle_r \nonumber\\
	& + \frac{1}{2}\int_s^{s+\epsilon}\left\{\frac{\partial f^0}{\partial e_k}(r,\cdots,X^k_r+,\cdots,X^n_r) \right.\nonumber\\
	&\left. -\frac{\partial f^0}{\partial e_k}(r,\cdots,X^k_r-,\cdots,X^n_r)\right\}I_{\{X^k_r=p^k_r\}}\md L^{(s,y,X^k-p^k,0)}_{r},\label{localtimeformula}
\end{align}
where
\[
L^{(s,y,Y,0)}_t= \mP^{s,y}-\lim_{\eta\to 0}\frac{1}{2\eta}\int_s^t I_{\{|Y_r|< \eta\}}\md \langle Y\rangle_r
\]
is the local time of semimartingale $Y$ at the point 0. For simplicity, we write $\frac{\partial f}{\partial e_k}(s,y^1,\cdots,y^k+,\cdots,y^n)= \frac{\partial f}{\partial e_k}(s,y+)$, and $\frac{\partial f}{\partial e_k}(s,y-)$ is defined similarly. Taking expectation with respect to $\mE^{s,y}$, other terms, except the last term, are of order $O(\epsilon)$, uniformly for $(s,y)\in B_{\delta}$, because $f_t$, $f_x$ and $f_{xx}$ are locally bounded, and the $\md W$ term in $\md X^i_r$ is martingale. As $\frac{\partial \tilde{f}}{\partial e_k}$ and $\frac{\partial g}{\partial e_k}$ are continuous, if we choose $\delta$ sufficiently small, we can assume $\frac{\partial f}{\partial e_k}(s,y+)-\frac{\partial f}{\partial e_k}(s,y-)\geq \delta_0>0$ for $(s,y)\in B_{\delta}\cap \partial C$. Therefore,
\begin{align}
	&\mE^{t,x}\int_t^{t+\epsilon}(\frac{\partial f^0}{\partial e_k}(r,X_r+)-\frac{\partial f^0}{\partial e_k}(r,X_r-))I_{\{X^k_r=p^k_r\}}\md L^{(t,x,X^k-p^k,0)}_{r}\nonumber \\&= \mE^{t,x}\int_t^{t+\epsilon}(\frac{\partial f}{\partial e_k}(r,X_r+) -\frac{\partial f}{\partial e_k}(r,X_r-))\chi(r,X_r)I_{\{X^k_r=p^k_r\}}\md L^{(t,x,X^k-p^k,0)}_{r}\nonumber\\
	& \geq \delta_0 \mE^{t,x}L^{(t,x,X^k-p^k,0)}_{t+\epsilon}-\int_t^{t+\epsilon}|1-\chi(r,X_r)|\md L^{(t,x,X^k-p^k,0)}_{r}.\label{infest1}
\end{align}
Here we have used the identity
\begin{align*}
	\frac{\partial f^0}{\partial e_k}(s',y'+)-\frac{\partial f^0}{\partial e_k}(s',y'-) &= (\frac{\partial f}{\partial e_k}(s',y'+)-\frac{\partial f}{\partial e_k}(s',y'-))\chi \\
	&+(\frac{\partial \chi}{\partial e_k}(s',y'+)-\frac{\partial \chi}{\partial e_k}(s',y'-))f \\
	&=(\frac{\partial f}{\partial e_k}(s',y'+)-\frac{\partial f}{\partial e_k}(s',y'-))\chi.
\end{align*}
On the other hand, because $\chi=1$ on $B_{\delta}$ and $d((s',y'),(s,y))<\delta/4,d((t,x),(s,y))<\delta/4 \Longrightarrow d((s',y'),(t,x))<\delta/2\Longrightarrow f^0=f,f^0_x=f_x,f^0_{xx}=f_{xx}$ at $(s',y')$, the last term in (\ref{infest1}) is not equal to 0 only if $\sup_{s\leq r\leq s+\epsilon}\|X_r-y   \|\geq \delta/4$. Therefore, it is bounded by
\begin{align*}
	C \mE^{s,y}L^{s,y,X^k-p^k,0}_{s+\epsilon}I_{\{\sup_{s\leq r\leq s+\epsilon}\|X_r-y   \|\geq \delta/4\}} & \leq C(\mE^{s,y}(L^{s,y,X^k-p^k,0}_{s+\epsilon})^2)^{1/2}\mP^{s,y}(\sup_{s\leq r\leq s+\epsilon}\|X_r-y   \|\geq \delta/4) ^{1/2} \\
	&\leq C(\mE^{s,y}(L^{s,y,X^k-p^k,0}_{s+\epsilon})^2)^{1/2}o(\epsilon).
\end{align*}
As a conclusion,
\begin{equation}\label{infest1'}
	\mE^{t,x}\int_t^{t+\epsilon}(\frac{\partial f^0}{\partial e_k}(r,X_r+)-\frac{\partial f^0}{\partial e_k}(r,X_r-))I_{\{X^k_r=p^k_r\}}\md L^{(t,x,X^k-p^k,0)}_{r}\geq \delta_0 \mE^{t,x}L^{(t,x,X^k-p^k,0)}_{t+\epsilon}-o(\epsilon).
\end{equation}
After direct calculation, for the semimartingale $Y=X^k-p^k$, we have
\[
\md Y_r=Q_r\md r+[\Lambda^k(r,X_r)-\sum_{l\neq k}\frac{\partial p}{\partial x_l}(r,X_r)\Lambda^l(r,X_r)]\md W_r,
\]
where $Q_r$ is induced by the first order terms in the It$\hato$'s formula, and is irrelevant to the proof.
Thus,
\[
\lim_{\epsilon\to 0}\mE\langle Y \rangle_{t+\epsilon}/\epsilon =\|\Lambda^k(t,x)-\sum_{l\neq k}\frac{\partial p}{\partial x_l}(t,x)\Lambda^l(t,x)\|^2.
\]
Note that $\Lambda^k(t,x)-\sum_{l\neq k}\frac{\partial p}{\partial x_l}(t,x)\Lambda^l(t,x)=\Lambda(t,x)\xi$, where $\xi_k=1$, hence $\xi\neq 0$. Using nondegenerency of $\Lambda$, we conclude that $Y$ satisfies (\ref{nondege}). Clearly it also satisfies (\ref{momentest}) by similar moment estimates of $X$. Based on Lemma \ref{inflm},
\begin{equation}
	\label{infest2}
	\lim_{\epsilon\to 0}\mE^{t,x}|X^k_{t+\epsilon}-p^k_{t+\epsilon}|/\epsilon=\infty.
\end{equation}
On the other hand, using It$\hato$-Tanaka's formula,
\[
|X^k_{t+\epsilon}-p^k_{t+\epsilon}|=\int_t^{t+\epsilon}\sgn(X^k_r-p^k_r)\md(X^k_r-p^k_r)+L^{(t,x,X^k-p^k,0)}_{t+\epsilon}.
\]
Thus, we have
\[
\mE^{t,x}L^{(t,x,X^k-p^k,0)}_{t+\epsilon}/\epsilon=\mE^{t,x}|X^k_{t+\epsilon}-p^k_{t+\epsilon}|/\epsilon+O(1).
\]
Combining (\ref{localtimeformula}), (\ref{infest1'}), (\ref{infest2}) and the above relation, we get
\[
\limsup_{\epsilon\to 0}\frac{\mE^{t,x}[f(t+\epsilon,X_{t+\epsilon})-f(t,x)]}{\epsilon}=\infty,
\]
which is contradicted to (\ref{sys3}).	\qedhere
\end{proof}

\section{On the assumptions (\ref{fcond}) and (\ref{fcon2})}
\label{fconddiscussion}
In some specific example where the analytical form is attainable (as in Section \ref{exm}), the assumptions (\ref{fcond}) and (\ref{fcon2}) can be verified directly. In this appendix, we also show that they are generally not very restrictive assumptions, and are satisfied once $\partial C$ and $g$ have sufficient regularity. We choose to use the theory of linear parabolic PDE to show the regularity of $f$, instead of the pure probabilistic approach as in \citet*{Friedman1975} and \citet*{He2019}. The main reason is that, their probabilistic approach is applicable based on the observation that the terminal time $T$ is fixed in the control setting. Because we consider stopping-control problems, the regular dependence of stopping times to the state variables shall also be considered, which makes the probabilistic approach quite complicated. Before proceeding, we first introduce and recall some notations that will be used in this appendix, as well as further assumptions to $g$ and $C$.

In this appendix, we will fix an arbitrary $(s,y)\in E$ and simply denote $f=f(s,y,\cdot,\cdot)$ as well as $g=g(s,y,\cdot,\cdot)$. We will also suppress any thing about $\hat{u}$ because it is fixed. For simplicity, we only consider $\mX=\mR^n$. We define the parabolic ball $B_{\delta}(t,x)=(t,t+\delta^2)\times \{x':\|x-x'\|<\delta\}$, and for any domain $\Omega\subset \mR^{n+1}$, define its parabolic boundary $\Pc \Omega$ by the points $(t,x)$ on $\partial \Omega$ (which is the topological boundary) such that for any $\delta>0$, $B_{\delta}(t,x)$ has points that are not in $\Omega$ (see \citet*{Lieberman1986a} for some specific examples). Let $B_k=\{ \|x\|<k \}\subset \mR^n$ and $E_k=[0,T)\times B_k$. We also use the standard notations for functions and surfaces with H$\ddo$lder continuity $C^{k+\alpha}$ for integer $k\geq 0$ and $\alpha\in (0,1)$. Specific definitions of these spaces are lengthy and far away from the topic of this paper. We refer readers to \citet*{Lieberman1996}. Now we are ready to state the assumptions that are needed in this appendix.
\begin{assumption}
\label{appasump}
  \begin{itemize}
    \item[(1)] $\Pc C\in C^{2+\alpha}$, $g\in C_{\rm loc}^{2+\alpha}(E)\cap L^{\infty}_{\rm poly}$.
    \item[(2)] For any sufficiently large integer $n$, there exists bounded domain $C_k\subset E_{2k}$, $\Pc C_k\in C^{2+\alpha}$ and $C_k \cap E_k= C \cap E_k$.
    \item[(3)] { Under $\hat{u}$, the operator $A$ is uniformly parabolic in $E$ and all its coefficients ($\Lambda$ and $\Theta$) have bounded weighted H$\ddo$lder seminorms in $E$ (see (5.15) on pages 94 in \citet*{Lieberman1996} for specific definition\big). Moreover, there is modular of continuity function $\zeta$ such that}
        \[
        \|\Lambda(t,x)-\Lambda(t',x')\|\leq \zeta(d((t,x),(t',x'))),\forall (t,x),(t',x')\in E
        \]   
  \end{itemize}
\end{assumption}
\begin{remark}
  Assumption \ref{appasump} is not restrictive. (1) and (3) are standard assumptions in PDE theory. (2) is true if $C$ is itself bounded or diffeomorphic to half-space. See Figure \ref{appendixfigure} for a graphical illustration.
\end{remark}
\begin{figure}[!htbp]
    \centering
    \def\svgwidth{\columnwidth}
    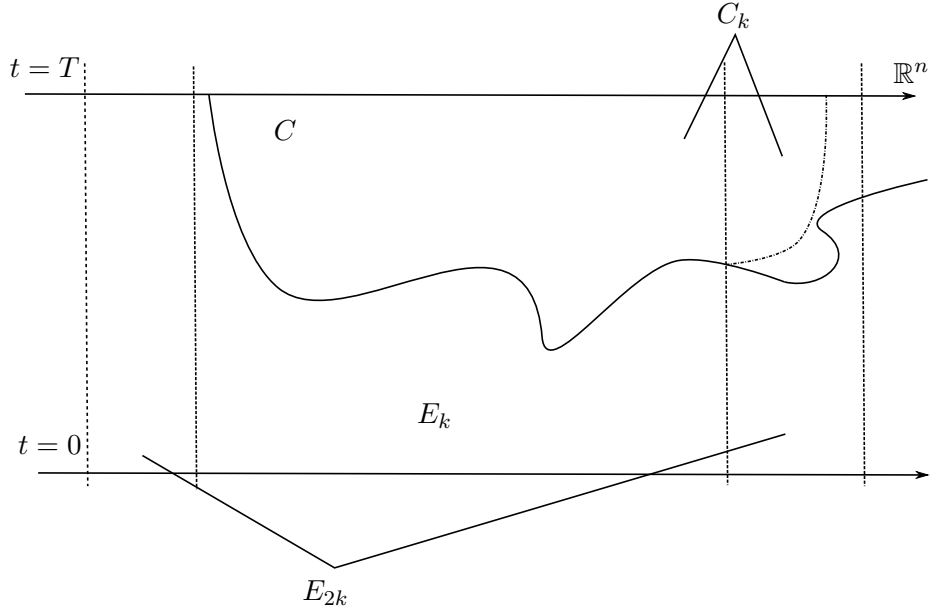
    \caption{An example of (2) in Assumption \ref{appasump}.}
    \label{appendixfigure}
\end{figure}
Before proceeding with the regularity of $f$, we first cite from \citet*{Lieberman1996} one classical result from theory of linear parabolic PDE on general bounded non-cylindrical domain.
\begin{lemma}[Theorem 5.15 from \citet*{Lieberman1996}]
  \label{parapde}
  Under Assumptions \ref{assumption1} and \ref{appasump}, there exists $f^k\in C^{1+\alpha,2+\alpha}(C_k)$ solving the following initial boundary value problem:
  \begin{equation}\label{initialboundary}
    \left\{
    \begin{array}{ll}
      Af^k=0, & {\rm in\ }C_k \\
      f^k=g, &  {\rm on\ }\Pc C_k.
    \end{array}
    \right.
  \end{equation}
\end{lemma}
Combining Lemma \ref{parapde} with Theorem \ref{representation} (note that here $f^k$ is bounded), we have
\[
f^k(t,x)=\mE^{t,x}g(\tau_k,X_{\tau_k}),\forall (t,x)\in C_k,
\]
where $\tau_k=\inf\{t\leq s\leq T: (s,X_s)\notin C_k     \}$. Now we prove the main result of this appendix.
\begin{lemma}
\label{applm888}
  Fix $N_0$ large enough. For any compact $K\subset C_{N_0}$, $f^k\to f$ uniformly on $K$ as $k\to \infty$.
\end{lemma}
\begin{proof}
  For any $(t,x)\in K$, noting $C_k \cap E_k= C \cap E_k$, we have
  \begin{align*}
    |f(t,x)-f^k(t,x)| &\leq \mE^{t,x}|g(\hat{\tau},X_{\hat{\tau}})-g(\tau_k,X_{\tau_k})|\\
     &\leq \mE^{t,x}\sup_{t\leq s \leq T}|g(s,X_s)|I_{\{\hat{\tau}\neq \tau_k\}} \\
     &\leq C(1+\|x\|^{\alpha})\mP^{t,x}(\exists t\leq s\leq T,X_s\notin E_k) \\
     &\leq C(1+\|x\|^{\alpha})\mP^{t,x}(\sup_{t\leq s\leq T}\|X_s-x\| \geq k-|x|).
  \end{align*}
  Similar to the proof of Lemma \ref{lm2}, we have
  \[
  \sup_{(t,x)\in K}|f-f^k|\leq C_{K,\alpha,\beta} k^{-\beta},
  \]
  where $\alpha$ and $\beta>0$ are constants independent of $k$. The proof is then completed.
\end{proof}
{ From Lemma \ref{applm888} we know $f^k$ converge to $f$ on any compact subset of $C$. From the interior H$\ddo$lder estimate of parabolic equations (see Theorem 4.9 of \citet*{Lieberman1996}), we know that $[f^k]_{2+\alpha,K}$ are bounded. Here, we emphasize that the estimates are only dependent on the parabolic operator and the distances between $K$ and $C_k$. For the former dependence, (3) in Assumptions \ref{appasump} guarantees that they are uniform in $k$. For the latter, we can choose $k$ large enough and use $C\cap E_k=C_k\cap E_k$ to ensure dist$(K,\Pc C_k)=$dist$(K,\Pc C)$. Therefore the interior estimates are independent of $k$, thus $[f^k]_{2+\alpha,K}$ are uniformly bounded.  Arzala-Ascoli theorem then yields $\partial_tf^k$, $\partial_x f^k$ and $\partial_{xx}f^k$ are all uniformly converging.}
 We thus know $f\in C^{1,2}(K)$ for any compact $K\subset C_{N_0}$. By arbitrary of $K$ and $N_0$, we conclude $f\in C^{1,2}(C)$. On the other hand, $f\in L^{\infty}_{\rm poly}$ comes naturally from its definition and $g\in L^{\infty}_{\rm poly}$. Thus, we have proved that (\ref{fcond}) is satisfied under Assumptions \ref{assumption1} and \ref{appasump}. As for (\ref{fcon2}), from H$\ddo$lder continuity of $f^k$ in $C_k$ we know that $f^k$, its first order time derivatives and its space derivatives up to order 2 are all uniformly continuous in $C_k$. Therefore, elementary analysis shows that all these functions can be extended to the boundary $\partial C_k$. Then an application of Whitney's extension theorem \big(see e.g. \citet*{Whitney1934} and \citet*{Seeley1973}\big) implies that $f^k$ extends to $f^k\in C^{1,2}(E)$ (detailed proof could be lengthy, technical and irrelevant to the topic of the present paper). Now similar arguments as in the uniform convergence mentioned above implies that we can get a limit function $\tilde{f}\in C^{1,2}(E)$. It is clear that $\tilde{f}$ is an extension of $f$.
\section{Proof of Proposition \ref{exmconclusion}}
\label{proofexmcon}
This appendix provides proof of Proposition \ref{exmconclusion}, which is lengthy and relies on several technical lemmas. Throughout this appendix, we will implicitly assume the assumptions (\ref{x0assumption}), (\ref{hassumption1}) and (\ref{hassumption2}). To exclude degenerate case we also assume $h'>0$. As a preliminary, we prove the well-posedness of the notations $x^*$, $x_0^*$ as well as the right hand side of (\ref{hassumption2}).
\begin{lemma}
\label{applm1}
  (\ref{exmsmoothfit}) has a unique solution on $(0,\infty)$. Setting $h=0$, the following equation also has a unique solution on $(0,\infty)$:
  \begin{equation}\label{exmsmoothfith0}
    \tag{\ref{exmsmoothfit}'}
    \alpha = (ax+\alpha)e^{-a(x-k)}.
  \end{equation}
\end{lemma}
\begin{proof}
  Consider $\psi(x)\triangleq -\alpha + (ax+\alpha)e^{-a(x-h(x)-k)}$. Direct calculation shows
  \[
  \psi'(x) = a(1-(1-h'(x))(ax+\alpha))e^{-a(x-h(x)-k)}.
  \]
  Clearly, for some $x'>0$, $\psi'>0$ on $(0,x')$ and $\psi'<0$ on $(x',\infty)$. Here we use the facts $h''(x)\leq 0$ and $h'(x)<\frac{1}{2}$. Now $\psi(0)=-\alpha+\alpha e^{ak}>0$, and
  \[
  x-h(x)= \int_0^x (1-h'(y))\md y \geq \frac{1}{2}x.
  \]
  Therefore, $\psi(x)\to -\alpha <0$ as $x\to \infty$. We conclude that there exists a unique $x^*$ such that $\psi(x^*)=0$. Simply letting $h=0$ in the same arguments gives the uniqueness of $x_0^*$.
\end{proof}
\begin{remark}
\label{psi0}
  As a byproduct of this proof, we have $\psi(x)>0$ on $(0,x^*)$.
\end{remark}
\begin{lemma}
  $x_0^*>\frac{2-\alpha}{a}$ if and only if $2e^{\alpha-2+ak}>\alpha$.
\end{lemma}
\begin{proof}
  Let $\psi_0(x)=-\alpha + (ax+\alpha)e^{-a(x-k)}$. From the proof of Lemma \ref{applm1} we know $x_0^*>\frac{2-\alpha}{a}$ if and only if $\psi_0(\frac{2-\alpha}{a})>0$, being equivalent to $2e^{\alpha-2+ak}>\alpha$.
\end{proof}
For convenience we denote
\[
F(x) = x^{\alpha}(1-e^{-a(x^*-h(x)-k)})-(x^*)^{\alpha}(1-e^{-a(x-h(x)-k)}).
\]
Then (\ref{fgeqg}) is equivalent to $F(x)\geq 0,\forall 0<x<x^*$. The proof of this statement will be accomplished by combining the following Lemmas \ref{applm10} and \ref{applm2}.
\begin{lemma}
\label{applm10}
   $F(x)> 0$ holds for $\frac{2-\alpha}{a}\leq x<x^*$.
\end{lemma}
\begin{proof}
  Consider the family of functions parameterized by $0<x<x^*$:
  \[
  \Phi(y;x) = y^{-\alpha}(1-e^{-a(y-h(x)-k)}).
  \]
  Then $F(x)\geq (>)0$ if and only if $\Phi(x^*;x)\geq (>)\Phi(x;x)$. Moreover, $\Phi'(y;x)=y^{-\alpha-1}\phi(y;x)$ with
  \[
  \phi(y;x) = -\alpha+(ay+\alpha)e^{-y-h(x)-k}.
  \]
  Using similar arguments as the proof of Lemma \ref{applm1}, we find $\phi'(y;x)>0$ for $y<\frac{1-\alpha}{a}$, $\phi'(y;x)<0$ for $y>\frac{1-\alpha}{a}$, $\phi(x;x)=\psi(x)>0$ for $0<x<x^*$ (see Remark \ref{psi0}), and
  \begin{align*}
  \phi(x^*;x)&=-\alpha+(ax^*+\alpha)e^{-a(x^*-h(x)-k)} \\
             &<-\alpha+(ax^*+\alpha)e^{-a(x^*-h(x^*)-k)} \\
             &=0
  \end{align*}
for $0<x<x^*$. Therefore, there exists a unique $\max\{\frac{1-\alpha}{a},x\}<\xi(x)<x^*$ such that $\phi(\xi(x);x)=0$. Based on implicit function theorem and the chain rule, we have
\[
\xi'(x)=\frac{h'(x)(a\xi(x)+\alpha)}{a\xi(x)-1+\alpha}>0.
\]
From $\phi(\xi(x^*),x^*)=\phi(x^*,x^*)=0$ we know $\xi(x^*)=x^*$. As a consequence, $x^* = \xi(x^*)>\xi(0)=x_0^*>\frac{2-\alpha}{a}$. We now have already known $\phi(y;x)>0$ for $x<y<\xi(x)$ and $\phi(y;x)<0$ for $\xi(x)<y<x^*$. Elementary calculus shows
\begin{align*}
 \Phi(x^*;x)-\Phi(x;x) &= \int_x^{x^*}y^{-\alpha-1}\phi(y;x) \md y\\
   & \geq \int_x^{\xi(x)}\xi(x)^{-\alpha-1}\phi(y;x)\md y +\int_{\xi(x)}^{x^*}\xi(x)^{-\alpha-1}\phi(y;x)\md y\\
   &= \xi(x)^{-\alpha-1}\int_x^{x^*}\phi(y;x)\md y.
\end{align*}
Now we only need to show
\[
\int_x^{x^*}\phi(y;x)\md x\geq 0, \frac{2-\alpha}{a}\leq x<x^*.
\]
To see this, using
\[
\phi''(y;x) =-a^2(2-ay-\alpha)e^{-a(y-h(x)-k)},
\]
if $x\geq \frac{2-\alpha}{a}$, then $y>\frac{2-\alpha}{a}$ for $x<y<x^*$ and $\phi(\cdot;x)$ is convex on $(x,x^*)$. Combining this with the fact that $\phi(\cdot;x)$ is decreasing on $(x,x^*)$, we have
\begin{align*}
\int_x^{x^*}\phi(y;x)\md y&= \int_x^{\xi(x)}\phi(y;x)\md y-\int_{\xi(x)}^{x^*}(-\phi(y;x))\md y \\
&\geq \frac{1}{2}(\xi(x)-x)^2|\phi'(\xi(x);x)|-\frac{1}{2}(x^*-\xi(x))|\phi'(\xi(x);x)| \\
&=\frac{1}{2}|\phi'(\xi(x);x)|[(\xi(x)-x)^2-(x^*-\xi(x))^2].
\end{align*}
We finally prove $\xi(x)-x \geq x^*-\xi(x)$ for $\frac{2-\alpha}{a}<x<x^*$, which will complete the proof. Indeed, denote $\triangle(x)= 2\xi(x)-x-x^*$, based on (\ref{hassumption2}), we have
\begin{align*}
\triangle'(x)&=2\xi'(x)-1 \\
             &=2h'(x)(1+\frac{1}{a\xi(x)+\alpha-1})-1\\
             &<2h'(x)(1+\frac{1}{a\xi(0)+\alpha-1})-1\\
             &=2h'(x)(1+\frac{1}{ax_0^*+\alpha-1})-1\\
             &\leq 0.
\end{align*}
Thus, $\triangle(x^*)=0$ leads to the desired conclusion.
\end{proof}
\begin{lemma}
\label{applm2}
  If $F(x)=0$ for some $0<x<\frac{2-\alpha}{a}$, then $F'(x)<0$.
\end{lemma}
\begin{proof}
  Slightly abusing the notation, we assume that there is a $x$ such that $0<x<\frac{2-\alpha}{a}$ and
  \[
  x^{\alpha}(1-e^{-a(x^*-h(x)-k)})=(x^*)^{\alpha}(1-e^{-a(x-h(x)-k)}).
  \]

Direct computation yields,
\begin{align*}
  F'(x) &=\alpha x^{\alpha-1}(1-e^{-a(x^*-h(x)-k)})-ax^{\alpha}h'(x)e^{-a(x^*-h(x)-k)}-a(x^*)^{\alpha}(1-h'(x))e^{-a(x-h(x)-k)}\\
  &=h'(x)\alpha x^{\alpha-1}(1-e^{-a(x^*-h(x)-k)})+(1-h'(x))\frac{(x^*)^{\alpha}}{x}(1-e^{-a(x-h(x)-k)})\\
  &-ax^{\alpha}h'(x)e^{-a(x^*-h(x)-k)}-a(x^*)^{\alpha}(1-h'(x))e^{-a(x-h(x)-k)}\\
  &=h'(x)x^{\alpha-1}(\alpha-(ax+\alpha)e^{-a(x^*-h(x)-k)})\\
  &+(1-h'(x))\frac{(x^*)^{\alpha}}{x}(\alpha-(ax+\alpha)e^{-a(x-h(x)-k)}).
\end{align*}
Therefore, $F'(x)<0$ is equivalent to
\begin{equation}
\label{app9}
h'(x)x^{\alpha-1}(\alpha-(ax+\alpha)e^{-a(x^*-h(x)-k)})
                 <(1-h'(x))\frac{(x^*)^{\alpha}}{x}(-\alpha+(ax+\alpha)e^{-a(x-h(x)-k)}).
\end{equation}
Clearly,
\[
\alpha-(ax+\alpha)e^{-a(x^*-h(x)-k)}>\alpha-(ax^*+\alpha)e^{-a(x^*-h(x^*)-k)}=0,
\]
and
\[
-\alpha+(ax+\alpha)e^{-a(x-h(x)-k)}=\psi(x)>0.
\]
Thus, one sufficient condition for (\ref{app9}) is
\[
h'(x)(\alpha-(ax+\alpha)e^{-a(x^*-h(x)-k)})<(1-h'(x))(-\alpha+(ax+\alpha)e^{-a(x-h(x)-k)}),
\]
being equivalent to
\[
\alpha<(ax+\alpha)e^{-a(x-h(x)-k)}[h'(x)(e^{-a(x^*-x)}-1)+1].
\]
Using the properties of $\psi$ (see the proof of Lemma \ref{applm1}), we know
\[
(ax+\alpha)e^{-a(x-h(x)-k)}> \min\{ \alpha e^{ak},2e^{\alpha-2+ak}\}, 0<x<\frac{2-\alpha}{a}.
\]
Using the assumption (\ref{hassumption2}), we find
\begin{align*}
(ax+\alpha)e^{-a(x-h(x)-k)}[h'(x)(e^{-a(x^*-x)}-1)+1]&>\min\{ \alpha e^{ak},2e^{\alpha-2+ak}\}[h'(x)(e^{-ax^*}-1)+1]\\
&>\alpha.
\end{align*}
The proof is thus completed.
\end{proof}
We are now ready to give the proof of Proposition \ref{exmconclusion}.
\begin{proof}[Proof of Proposition \ref{exmconclusion}]
  Based on Lemma \ref{applm1}, we only need to prove $F(x)\geq 0$ for $0<x<\frac{2-\alpha}{a}$. If not, we assume that there is $0<x'<\frac{2-\alpha}{a}$ such that $F(x')<0$. By continuity and the fact $F(\frac{2-\alpha}{a})>0$, $F$ has roots in $(x',\frac{2-\alpha}{a})$ based on mid-value theorem. Define
  \[
  \hat{x}=\inf\{x'<x<\frac{2-\alpha}{a}:F(x)=0\}.
  \]
  Clearly, $F(\hat{x})=0$ so that $\hat{x}>x'$. Using Lemma \ref{applm2}, we know $F'(\hat{x})<0$. Thus, for a sufficiently small $\epsilon$, $F(x)>0$ for $\hat{x}-\epsilon<x<\hat{x}$. Using mid-value theorem again, we can find a $\tilde{x}<\hat{x}-\epsilon$ such that $F(\tilde{x})=0$, contradicting to the definition of $\hat{x}$. The proof is completed.
\end{proof}
{ 
\section{Proof of Lemma \ref{sBVPthm}}
\label{extsBVP}
We first study the existence of a solution to (\ref{sBVP1}). To do this, we need to consider a properly introduced auxiliary problem to reduce the singularity and then apply Leray-Schauder topological degree theory and Green's function. To be specific, we consider a coordinate transform $t \mapsto s=\sqrt{t}$, and $y(s)=h(t)=h(s^2)$. Now chain rule implies $h'(t)=y'(s)/2s$, $h''(t)=y''(s)/4s^2-y'(s)/4s^3$. Therefore the singular boundary value problem of $y$ reads:
\begin{equation}
	\label{sBVP1'}\tag{\ref{sBVP1}'}
	\left\{
	\begin{aligned}
		&y''=\frac{y'}{s}-2\kappa\left[\frac{\beta_1\beta_2s^4}{y^2}-\frac{(\beta_1+\beta_2)s^2}{y}+\frac{y's}{2y}\right],\\
		&y(0)=0, y(\sqrt{2})=\beta_1+\beta_2.
	\end{aligned}
	\right.
\end{equation}
Because the coordinate transform we apply is one-to-one from $[0,2]$ to $[0,\sqrt{2}]$, the existence of a solution to (\ref{sBVP1'}) will give a the existence of a solution to (\ref{sBVP1}). For the estimate $\beta_1\leq h'\leq \beta_2$, we only need to establish the estimate $2\beta_1 s\leq y'(s)\leq 2\beta_2 s$, $\forall s\in [0,\sqrt{2}]$. We give the following result about (\ref{sBVP1'}):
\begin{proposition}
	\label{sBVPsol}
	Problem (\ref{sBVP1'}) has a solution $y\in C^2[0,\sqrt{2}]$ such that $2\beta_1 s\leq y'(s)\leq 2\beta_2 s$, $\forall s\in [0,\sqrt{2}]$.
\end{proposition}
\begin{proof}
	We first consider two cut-off functions defined as follows:
	\[\chi_1(s,y)=
	\left\{
	\begin{aligned}
		&\beta_2 s^2, &\mif\ y>\beta_2 s^2,\\
		&\beta_1 s^2, &\mif\ y<\beta_1 s^2,\\
		&y,           &\ow,
	\end{aligned}	
	\right.
	\]
		\[\chi_2(s,z)=
	\left\{
	\begin{aligned}
		&2\beta_2 s, &\mif\ z>2\beta_2 s,\\
		&2\beta_1 s, &\mif\ z<2\beta_1 s,\\
		&z,           &\ow.
	\end{aligned}	
	\right.
	\]
	Define
	\[
	H(s,y,z)\triangleq \frac{\chi_2(s,z)}{s}-2\kappa\left[\frac{\beta_1\beta_2s^4}{\chi_1(s,y)^2}-\frac{(\beta_1+\beta_2)s^2}{\chi_1(s,y)}+\frac{\chi_2(s,z)s}{2\chi_1(s,y)}\right].
	\]
	We now consider the existence of a solution to the following `cut problem':
	 \begin{equation}
	 	\label{sBVP1''}\tag{\ref{sBVP1}''}
	 	\left\{
	 	\begin{aligned}
	 		&y''=H(s,y,y'),\\
	 		&y(0)=0, y(\sqrt{2})=\beta_1+\beta_2.
	 	\end{aligned}
	 	\right.
	 \end{equation}
 It is well known that { $y\in C^1\triangleq C^1[0,\sqrt{2}]$} is a solution to (\ref{sBVP1''}) if and only if it is a fixed point of the operator $T:C^1 \to C^1$, with
 \[
 (Ty)(s)=\int_0^{\sqrt{2}}G(s,r)H(r,y(r),y'(r))\md r+(\beta_1+\beta_2)s/{\sqrt{2}},
 \]
 where $G$ is the Green function on $[0,\sqrt{2}]$ with homogeneous Dirichlet  boundary condition. Because it has analytical form, we can easily conclude that both $G$ and $\partial_t G$ are bounded. Because of the cut-off, there is a constant $C=C(\beta_1,\beta_2)$ such that $|H(s,y,z)|\leq C$. Thus $Ty$ has bounded second order derivative on $[0,\sqrt{2}]$ and $T(C^1)\subset C^2$. By Azela-Ascoli theorem, $T$ is a compact operator on $C^1$, and there exists a $R$ which is large enough, such that $T(C^1)\subset B(0,R/2)$. { Here and afterwards in this proof, we denote by $B(0,R)$ the ball in $C^1$, centered at 0, with radius $R$.} Now consider $T^\alpha=\alpha T$ for $\alpha\in [0,1]$, we have $T^{\alpha}(C^1)\subset B(0,\alpha R/2)$. Thus for any $y\in \partial B(0,R)$ and $\alpha\in[0,1]$, $y\neq T^{\alpha}y$. By homotopy property of Leray-Schauder degree ({ c.f. Theorem 11.7 of \citet*{Brown1993}}), we have
 \[
 \deg(I-T,B(0,R))=\deg(I,B(0,R))=1.
 \]
 By degree theory, we have a fixed point $y\in C^2$, that is to say, a solution to (\ref{sBVP1''}). We then claim that
 \begin{equation}
 	\label{claim1}
 2 \beta_1 s\leq y'(s)\leq 2\beta_2 s,\forall s\in [0,\sqrt{2}].
 \end{equation}
 If (\ref{claim1}) is true, then integrating on $[0,s]$ gives $\beta_1 s^2\leq y(s) \leq s^2$. By definition of $\chi_1$ and $\chi_2$, (\ref{sBVP1''}) is reduced to (\ref{sBVP1'}), and the desired estimations hold naturally and this finishes the proof. To show (\ref{claim1}), we assume by contradiction that $y'(s_0)-2\beta_2 s_0=\max_{s\in [0,\sqrt{2}]}\{y'(s)-2\beta_2 s\}>0$. There are three cases to be discussed:
 \begin{itemize}
 	\item[(1)]$s_0=0$. By continuity, there exist a $\delta>0$ such that $y'(s)>2\beta_2 s$ for $s\in [0,\delta]$. Define $\delta^*=\sup\{\delta\in (0,\sqrt{2}]:y'(s)>2\beta_2 s,\forall s\in [0,\delta] \}$. Then $0<\delta^*\leq \sqrt{2}$. If $\delta^*<\sqrt{2}$, by integration we have $y(s)\geq\beta_2 s^2$ for $s\in [0,\delta^*]$, thus $\chi_1(s,y(s))=\beta_2 s^2$, $\chi_2(s,y'(s))=2\beta_2 s$, which gives $y''(s)=2\beta_2 s$. As a consequences, $y'(\delta^*)=y'(0)+2\beta_2 \delta^*>2\beta_2 \delta^*$. Again by continuity, the strict inequality can be extended to a larger $\delta'>\delta^*$, contradicting to the definition of $\delta^*$. Therefore $\delta^*=\sqrt{2}$, which leads to $y(s)=y'(0)s+\beta_2 s^2$ for $s\in [0,\sqrt{2}]$, but now $y(\sqrt{2})=\sqrt{2}y'(0)+2\beta_2>2\beta_2>\beta_1+\beta_2$ is contradicting to the boundary condition.
 	\item[(2).]$0<s_0<\sqrt{2}$. In this case we have $y''(s_0)=2\beta_2$. By the equation in (\ref{sBVP1''}) we have
 	\[
 	\frac{\beta_1 \beta_2 s_0^4}{\chi_1(s_0,y(s_0))^2}-\frac{(\beta_1 s_0^2)}{\chi_1(s_0,y(s_0))}=0.
 	\]
 	Therefore $\chi_1(s_0,y(s_0))=\beta_2 s_0^2$, which leads to $y(s_0)\geq \beta_2 s_0^2$. Now repeating the same argument as in Case (1) gives the lower bound $y(\sqrt{2})\geq 2\beta_2>\beta_1+\beta_2$, a contradiction.
 	\item[(3).]$s_0=\sqrt{2}$. In this case by boundary condition we can find a small $\epsilon$ such that $\chi_1(s,y(s))<\beta_2 s^2$ and $\chi_2(s,y'(s))=2\beta_2 s$ on $[\sqrt{2}-\epsilon,\sqrt{2}]$. This implies
 	\[
 	y''(s)-2\beta_2=-2\kappa\frac{\beta_1 s^2(\beta_2 s^2-\chi_1(s,y(s)))}{\chi_1(s,y(s))^2}<0.
 	\]
 	That is, $y'(s)-2\beta_2 s$ is decreasing on $[\sqrt{2}-\epsilon,\sqrt{2}]$, contradicting to the assumption that $y'(s)-2\beta_2 s$ attains its maximum at $s_0$.
 \end{itemize}
We have proved $y'(s)\leq 2\beta_2 s$ for any $s\in [0,\sqrt{2}]$. By symmetry, it is straightforward to show $y'(s)\geq 2\beta_1 s$ for any $s\in [0,\sqrt{2}]$. We have now proved (\ref{claim1}), and thus Proposition \ref{sBVPsol}.
\end{proof}
We now turn to problem (\ref{sBVP2}). $h$ is now known, and satisfies the estimation $\beta_1 t\leq h(t)\leq \beta_2 t$, for any $t\in [0,2]$.
\begin{proposition}
	\label{appDpro}
	For any fixed $b>0$, Problem (\ref{sBVP2}) has a positive solution $\psi^b$ which is strictly increasing on $(0,b)$.
\end{proposition}
\begin{proof}
	Define a function $u:[0,2]\to [0,b]$ as
	\[
	u(z)\triangleq \frac{b\int_0^z e^{-\frac{\kappa}{2}\int_{z'}^2\frac{1}{h(y)}\md y} \md z'}{\int_0^2 e^{-\frac{\kappa}{2}\int_{z'}^2\frac{1}{h(y)}\md y} \md z'}.
	\]
	Clearly, it is smooth and is strictly increasing and thus is one-to-one from $[0,2]$ to $[0,b]$. Moreover, it satisfies
	\begin{equation}
		\label{uequation}
	u''=\frac{\kappa}{2}\frac{1}{h(z)}u',
	\end{equation}
	as well as the boundary condition $u(0)=0$, $u(2)=b$. We further claim that
	\begin{equation}
		\label{claim2}
		(z/2)^{1/\alpha_1}\leq u(z)/b\leq (z/2)^{1/\alpha_2}.
	\end{equation}
    To simplify our analysis, we denote
    \[
    P_i(z)=\frac{\int_0^z e^{-\frac{\kappa}{2}\int_{z'}^2\frac{1}{h(y)}\md y} \md z'}{z^{1/\alpha_i}},i=1,2.
    \]
    It is clear that one sufficient condition for (\ref{claim2}) is, { $P_1$} is decreasing and { $P_2$} is increasing. To verify this, we compute
    \begin{align*}
    P_i'(z)&=\frac{e^{-\frac{\kappa}{2}\int_{z}^2\frac{1}{h(y)}\md y}z-(1+\frac{\kappa}{2\beta_i})\int_0^z e^{-\frac{\kappa}{2}\int_{z'}^2\frac{1}{h(y)}\md y} \md z' }{z^{2+\frac{\kappa}{2\beta_i}}} \\
    &=\frac{\frac{\kappa}{2}\int_0^z\left( \frac{z'}{h(z')}-\frac{1}{\beta_i}\right)e^{-\frac{\kappa}{2}\int_{z'}^2\frac{1}{h(y)}\md y} \md z'  } {z^{2+\frac{\kappa}{2\beta_i}}}.
    \end{align*}
It is clear that the desired monotonicity is guaranteed by the estimation of $h$: $\beta_1 t\leq h(t)\leq \beta_2 t$. We now go back to (\ref{sBVP2}). Because $u$ is strictly increasing, we can denote by $\psi^b$ its inverse function, which is strictly increasing from $[0,b]$ to $[0,2]$ and satisfies boundary conditions of (\ref{sBVP2}). Moreover, by implicit function theorem, it is straightforward to derive the relations $(\psi^b)'=1/u'$, $(\psi^b)''=-u''/(u')^3$. Therefore (\ref{uequation}) implies that $\psi^b$ satisfies the equation of (\ref{sBVP2}). Finally, the estimation of $\psi^b$ is derived from (\ref{claim2}) by the correspondence $z\mapsto \psi^b(x)$, $x\mapsto u(z)$.
\end{proof}
}

\end{document}